\documentclass{amsart}
\usepackage{amsmath}
\usepackage{amssymb}
\usepackage{amsfonts,amsthm}
\usepackage{upref,exscale}
\usepackage{cite}

\theoremstyle{plain}
\newtheorem{theorem}{Theorem}[section]
\newtheorem{corollary}[theorem]{Corollary}
\newtheorem{lemma}[theorem]{Lemma}
\newtheorem{proposition}[theorem]{Proposition}
\theoremstyle{definition}

\newtheorem*{acknowledgement}{Acknowledgement}
\newtheorem{claim}{Claim}

\numberwithin{equation}{section}

\begin{document}
\title[Alternating sign multibump solutions]{Alternating sign multibump solutions of nonlinear elliptic equations in
expanding tubular domains}
\author{Nils Ackermann}
\address{Instituto de Matem\'{a}ticas, Universidad Nacional Aut\'{o}noma de M\'{e}xico,
Circuito Exterior, C.U., 04510 M\'{e}xico D.F., Mexico.\smallskip}
\author{M\'{o}nica Clapp}
\address{Instituto de Matem\'{a}ticas, Universidad Nacional Aut\'{o}noma de M\'{e}xico,
Circuito Exterior, C.U., 04510 M\'{e}xico D.F., Mexico.\smallskip}
\author{Filomena Pacella}
\address{Dipartimento di Matematica, Universit\`{a} "La Sapienza" di Roma, P.le A. Moro
2, 00185 Roma, Italia}
\thanks{This research was partially supported by CONACYT grant 129847 and
PAPIIT-DGAPA-UNAM grants IN101209 and IN106612 (Mexico), and by exchange funds
of the Universit\`{a} \textquotedblleft La Sapienza\textquotedblright\ di Roma (Italy).}
\date{October, 2012}

\begin{abstract}
Let $\Gamma$ denote a smooth simple curve in $\mathbb{R}^{N}$, $N\geq2$,
possibly with boundary. Let $\Omega_{R}$ be the open normal tubular
neighborhood of radius $1$ of the expanded curve $R\Gamma:=\{Rx\mid x\in
\Gamma\smallsetminus\partial\Gamma\}$. Consider the superlinear problem
$-\Delta u+\lambda u=f(u)$ on the domains $\Omega_{R}$, as $R\rightarrow
\infty$, with homogeneous Dirichlet boundary condition. We prove the existence
of multibump solutions with bumps lined up along $R\Gamma$ with alternating
signs. The function $f$ is superlinear at $0$ and at $\infty$, but it is not
assumed to be odd.

If the boundary of the curve is nonempty our results give examples of
contractible domains in which the problem has multiple sign changing solutions.

\end{abstract}
\maketitle

\section{Introduction}

\label{sec:introduction}

Let $\gamma\in C^{3}([0,1],\mathbb{R}^{N})$, $N\geq2$, be a curve without
self-intersections except possibly for $\gamma(0)=\gamma(1)$. In this case we
also assume that $\dot{\gamma}(0)=\dot{\gamma}(1)$. For $R>0$ define
\begin{equation}
\Omega_{R}:=\text{int}{\bigcup_{t\in\lbrack0,1]}}\{R\gamma(t)+v\mid
v\in\mathbb{R}^{N},\ \lvert v\rvert<1,\ \dot{\gamma}(t)\cdot v=0\},
\label{eq:72}
\end{equation}
where int$(X)$ denotes the interior of $X$ in $\mathbb{R}^{N}$. Thus, for $R$
large enough, $\Omega_{R}$ is the tubular neighborhood of radius $1$ of the
$1$-dimensional submanifold $\Gamma_{R}$ of $\mathbb{R}^{N}$ defined as
\begin{equation*}
\Gamma_{R}:=
\begin{cases}
\{R\gamma(t)\mid t\in\lbrack0,1]\}, & \qquad\text{if }\gamma(0)=\gamma(1),\\
\{R\gamma(t)\mid t\in(0,1)\}, & \qquad\text{if }\gamma(0)\neq\gamma(1).
\end{cases}
\end{equation*}
We are interested in finding solutions to the problem
\begin{equation}
\left\{
\begin{array}
[c]{ll}
-\Delta u+\lambda u=f(u) & \text{in }\Omega_{R},\\
u=0 & \text{on }\partial\Omega_{R},
\end{array}
\right.  \label{eq:1}
\end{equation}
for $R$ large enough.

Let $\lambda_{1,1}$ be the first eigenvalue of the Laplace operator $-\Delta$
in the unit ball in $\mathbb{R}^{N-1}$ with Dirichlet boundary conditions. Set
$p_{S}:=\infty$ if $N=1,2$ and $p_{S}:=(N+2)/(N-2)$ if $N\geq3$. We make the
following assumptions:

\begin{enumerate}
\item[(H1) ] $\lambda>-\lambda_{1,1}$.

\item[(H2) ] $f\in C^{1}(\mathbb{R})\cap C^{3}(\mathbb{R}\backslash\{0\})$.

\item[(H3)] There are $C>0$ and $p_{1},p_{2}\in(1,p_{S})$ such that $p_{1}\leq
p_{2}$ and
\begin{equation*}
\lvert f^{(k)}(u)\rvert\leq C(\lvert u\rvert^{p_{1}-k}+\lvert u\rvert
^{p_{2}-k})
\end{equation*}
for $k\in\{0,1,2,3\}$ and $u\neq0$.

\item[(H4)] $f(u)u>0$ for all $u\neq0$.
\end{enumerate}

Note that
\begin{equation}
f(0)=f^{\prime}(0)=0. \label{eq:50}
\end{equation}

For example, the standard nonlinearity $f(u):=\lvert u\rvert^{p-1}u$ satisfies
(H1)-(H4) if $p\in(1,p_{S})$.

We write a point in $\mathbb{R}^{N}$ as $(\xi,\eta)$, with $\xi\in\mathbb{R}$
and $\eta\in\mathbb{R}^{N-1}$, and denote the cylinder in $\mathbb{R}^{N}$ of
radius $1$ around the $\xi$-axis by
\begin{equation*}
\mathbb{L}:=\{(\xi,\eta)\in\mathbb{R}^{N}\mid\lvert\eta\rvert<1\}.
\end{equation*}
Locally, $\mathbb{L}$ is the limit domain of $\Omega_{R}$ as $R\rightarrow
\infty$. So we consider the limit problem
\begin{equation}
\left\{
\begin{array}
[c]{l}
-\Delta u+\lambda u=f(u),\\
u\in H_{0}^{1}(\mathbb{L)}.
\end{array}
\right.  \label{eq:2}
\end{equation}
By Lemma~\ref{lem:schroed-invertible} below, the operator $-\Delta+\lambda$
with Dirichlet boundary conditions in $L^{2}(\mathbb{L})$ has a positive
spectrum. If $f$ satisfies an Ambrosetti-Rabinowitz type condition the
mountain pass theorem, together with the translation invariance in the $\xi
$-direction and concentration compactness, yields a positive and a negative
solution to \eqref{eq:2}, having minimal energy in their respective cones. We
add the following assumption:

\begin{enumerate}
\item[(H5)] Problem \eqref{eq:2} has a positive solution $U^{+}$ and a
negative solution $U^{-}$ which are nondegenerate, in the sense that the
solution space of the linearized problem
\begin{equation*}
-\Delta u+\lambda u=f^{\prime}(U^{\pm})u,\qquad u\in H_{0}^{1}(\mathbb{L}),
\end{equation*}
has dimension one.
\end{enumerate}

Note that the solution space of the linearized problem\ must have at least
dimension one, due to the invariance under translations. Hypothesis (H5)
requires that these are the only elements in the kernel of the linearization.
This condition is not easy to check, even for the standard nonlinearity
$f(u):=u^{p}.$ For this $f,$ Dancer showed in \cite{MR1962054} that (H5) holds true
either for $\lambda=0$ and almost every $p\in(1,p_{S}),$ or for almost every
$\lambda\in(0,\infty)$ and every $p\in(1,p_{S}).$

By \cite[Theorem~1.2]{MR1395408} the solutions $U^{\pm}$ are radially symmetric in
$\eta$ and decreasing in $\lvert\eta\rvert$. Moreover, by \cite[Theorem~6.2]
{MR1039342}, after a translation in the $\xi$-direction, we may assume that they are
also even in $\xi$ and decreasing in $\left\vert \xi\right\vert $. It follows
that they have a unique extremal point at $0$. We extend $U^{\pm}$ to all of
$\mathbb{R}^{N}$ by setting them as $0$ outside of $\mathbb{L}$.

For each $x\in\Gamma_{R}$ we choose a linear isometry $A_{x}$ which maps the
tangent space of $\Gamma_{R}$ at $x$ onto $\mathbb{R}\times\{0\}$ and its
orthogonal complement onto $\{0\}\times\mathbb{R}^{N-1}$, and we define
\begin{equation}
U_{x,R}^{\pm}(y):=U^{\pm}(A_{x}(y-x))\qquad\text{for all }y\in\mathbb{R}^{N}.
\label{eq:3}
\end{equation}
Since $U^{\pm}$ is radially symmetric in $\xi$ and in $\eta$, the function
$U_{x,R}^{\pm}$ is independent of the choice of $A_{x}$.

The parametrization $\gamma$ induces an orientation on $\Gamma_{R}$ which
allows to give an order to every finite set of points in $\Gamma_{R}.$ We
shall say that $(x_{1},\ldots,x_{n})\in\left(  \Gamma_{R}\right)  ^{n}$ is an
$n$-\emph{chain} in $\Gamma_{R}$ if there exist $0\leq t_{1}<t_{2}
<\cdots<t_{n}<1$ such that
\begin{equation}
x_{i}=R\gamma(t_{i})\qquad\text{for }i=1,2,\ldots,n. \label{eq:33}
\end{equation}
If $\gamma(0)=\gamma(1)$ a circular shift $(x_{i},\ldots,x_{n},x_{1}
,\ldots,x_{i-1})$ of an $n$-chain will also be called an $n$-chain. We shall
prove the following results.

\begin{theorem}
\label{thm:no-boundary} Assume that $\gamma(0)=\gamma(1)$. Suppose also that
\emph{(H1)-(H5)} hold. For each $k\in\mathbb{N}$ there exists $R_{k}>0$ such
that for every $R\geq R_{k}$ there are a $2k$-chain $(x_{R,1},x_{R,2}
,\ldots,x_{R,2k})\in\left(  \Gamma_{R}\right)  ^{2k}$ and a solution $u_{R}$
of \eqref{eq:1} such that
\begin{equation}
u_{R}=\sum_{i=1}^{k}(U_{x_{R,2i-1},R}^{+}+U_{x_{R,2i},R}^{-})+o(1)
\label{eq:4}
\end{equation}
in $H^{1}(\mathbb{R}^{N})$ as $R\rightarrow\infty$. Moreover, $\lvert
x_{R,i}-x_{R,j}\rvert\rightarrow\infty$ as $R\rightarrow\infty$, if $i\neq j$.
\end{theorem}

\begin{theorem}
\label{thm:with-boundary} Assume that $\gamma(0)\neq\gamma(1).$ Suppose also
that \emph{(H1)-(H5)} hold. For each $n\in\mathbb{N}$, $n\geq2$, there exists
$R_{n}>0$ such that for every $R\geq R_{n}$ there are an $n$-chain
$(x_{R,1},x_{R,2},\ldots,x_{R,n})\in\left(  \Gamma_{R}\right)  ^{n}$ and a
solution $u_{R}$ of \eqref{eq:1} such that
\begin{equation}
u_{R}=\sum_{i=1}^{k}(U_{x_{R,2i-1},R}^{+}+U_{x_{R,2i},R}^{-})+(n-2k)U_{x_{R,n}
,R}^{+}+o(1) \label{eq:5}
\end{equation}
in $H^{1}(\mathbb{R}^{N})$ as $R\rightarrow\infty$, where $k$ is the largest
integer smaller than or equal to $n/2.$ Moreover, as $R\rightarrow\infty$,
$\lvert x_{R,i}-x_{R,j}\rvert\rightarrow\infty$ if $i\neq j$, and
\emph{dist}$(x_{R,i},\partial\Gamma_{R})\rightarrow\infty$ for all $i $.
\end{theorem}

All solutions constructed in Theorems~\ref{thm:no-boundary} and
\ref{thm:with-boundary} change sign. If $\gamma$ is a closed curve these
solutions have an even number of bumps with alternating signs along the curve,
whereas in the open-end case $\gamma(0)\neq\gamma(1)$ the number of
alternating bumps may be even or odd. Note that the term $(n-2k)$ in
Theorem~\ref{thm:with-boundary} is $0$ if $n$ is even, and it is $1$ if $n$ is
odd. In the first case we have a positive bump at one end and a negative bump
at the other end of the domain, and in the second case we have positive bumps
at both ends. Of course, applying Theorem~\ref{thm:with-boundary} with $f(u)$
replaced by $-f(-u)$ and then multiplying the obtained multibump solution by
$-1$, we obtain a solution with negative bumps at both ends, as well.

Observe that in the open-end case the domains $\Omega_{R}$ are contractible,
and they are even convex if $\Gamma$ is a segment. This means that to get
multiplicity of sign changing solutions neither topological nor particular
geometrical assumptions are needed. This stands in contrast with the case of
positive solutions where it has been conjectured that for some power-type
nonlinearities only one positive solution exists in any convex
domain~\cite{MR949628}, as it does in a ball. Of course this difference between
multiplicity of positive and sign changing solutions can be easily understood
by looking at odd nonlinearities. In fact, if $f$ is odd (for example, if
$f(u)=\lvert u\rvert^{p-1}u$, $p\in(1,p_{S})$) it is well known that
infinitely many sign changing solutions exist in any bounded domain. Our
results \emph{do not assume that} $f$ \emph{is odd}, therefore multiplicity of
sign changing solutions is not so obvious. In fact, if $f$ is not odd only few
multiplicity results are available, see e.g. \cite{MR2037264, MR1627654}.

Dancer exhibited positive solutions with multiple bumps for \textquotedblleft
dumbbell shaped domains\textquotedblright\ \cite{MR949628,MR1072904}. Sign changing
solutions may also be constructed in domains of this type. On the other hand,
if $\Gamma$ is a segment, Theorem~\ref{thm:with-boundary} yields examples of
\emph{convex} domains in which problem \eqref{eq:1} has at least $k$ nodal
solutions with up to $k+1$ peaks, for any given $k$, without assuming that $f$
is odd. We believe this is the first result of this type.

As in other similar problems, the procedure to prove
Theorems~\ref{thm:no-boundary} and \ref{thm:with-boundary} is to consider
approximate solutions to problem \eqref{eq:1} and then show that near them a
true solution exists. So, to start, we need to make a good guess as to what
the approximate solutions should be. The geometry of our expanding domains
suggests looking at functions of the form
\begin{equation*}
U_{x_{R,1},R}^{+}+U_{x_{R,2},R}^{-}+U_{x_{R,3},R}^{+}+U_{x_{R,4},R}^{-}+\cdots
\end{equation*}
for finitely many points $x_{R,1},x_{R,2},x_{R,3},x_{R,4},...$, ordered along
the curve, whose number is even if the curve is closed. Then some estimates
are needed to show that these are indeed good approximate solutions and to
compute the order of the approximation. To prove the existence of a true
solution near them we follow a well-known Lyapunov-Schmidt reduction
procedure, which relies on the contraction mapping principle. This requires
again careful estimates on the approximate solutions and their linearization.
Finally, a critical point of the reduced problem is obtained by a
minimization. Here the crucial role is played by the fact that the interaction
between a positive and a negative bump increases the value of the energy
functional. This explains why the bumps should be placed along the tube with
alternating signs and why the number of bumps must be even in the closed tube
case (Theorem~\ref{thm:no-boundary}). In the open-end case
(Theorem~\ref{thm:with-boundary}) the energy also increases as a bump
approaches an end of the tube. Therefore, in both cases, a solution to the
reduced problem is obtained by minimizing the energy.

It is harder to prove similar results when $\Gamma$ is a higher dimensional
manifold, instead of a curve. For positive solutions some results were
obtained by Dancer and Yan \cite{MR1886955} when $\Gamma$ is the boundary of a convex
domain. Positive multibump solutions in a tubular neighborhood of an expanding
compact manifold have been constructed in \cite{aclapa-2}. The problem of
constructing sign changing solutions in such domains is more subtle and
requires minimax arguments.

The outline of the paper is as follows: In section~\ref{sec:prel-estim-1} we
have collected some tools, and results about the linear problem.
Section~\ref{sec:energy-estimates} contains the essential energy estimates,
while in section~\ref{sec:finite-dimens-reduct} we describe the finite
dimensional reduction and prove our main results.

\begin{acknowledgement}
Filomena Pacella wishes to thank the Mathematics Institute at UNAM and
Nils Ackermann and M\'{o}nica Clapp wish to thank the Mathematics Department
of the Universit\`{a} \textquotedblleft La Sapienza\textquotedblright\ di Roma
for their kind and warm hospitality.
\end{acknowledgement}

\section{Preliminaries}

\label{sec:prel-estim-1}

\subsection{Algebraic and geometric tools}

\label{sec:algebr-geom-tools}We start with some elementary lemmas which will
be used later to estimate the interactions.

\begin{lemma}
\label{lem:interaction-exponential} Suppose that $\mu_{k}>\bar{\mu}\geq0$ for
$k=1,2,3$. Then there is $C>0$ such that the inequalities
\begin{equation}
\int_{\mathbb{R}^{N}}\mathrm{e}^{-\mu_{1}\lvert x-x_{1}\rvert}\mathrm{e}
^{-\mu_{2}\lvert x-x_{2}\rvert}\,\mathrm{d}x\leq C\mathrm{e}^{-\bar{\mu}\lvert
x_{1}-x_{2}\rvert} \label{eq:19}
\end{equation}
and
\begin{equation}
\int_{\mathbb{R}^{N}}\mathrm{e}^{-\mu_{1}\lvert x-x_{1}\rvert}\mathrm{e}
^{-\mu_{2}\lvert x-x_{2}\rvert}\mathrm{e}^{-\mu_{3}\lvert x-x_{3}\rvert
}\,\mathrm{d}x\leq C\exp\biggl (-\bar{\mu}\min_{x\in\mathbb{R}^{N}}\sum
_{k=1}^{3}\lvert x-x_{k}\rvert\biggr ) \label{eq:20}
\end{equation}
hold true for all $x_{1},x_{2},x_{3}\in\mathbb{R}^{N}$.
\end{lemma}

\begin{proof}
Since $\bar{\mu}\lvert x_{1}-x_{2}\rvert+(\mu_{2}-\bar{\mu})\lvert
x-x_{2}\rvert\leq\bar{\mu}(\lvert x-x_{1}\rvert+\lvert x-x_{2}\rvert)+(\mu
_{2}-\bar{\mu})\lvert x-x_{2}\rvert\leq\mu_{1}\lvert x-x_{1}\rvert+\mu
_{2}\lvert x-x_{2}\rvert,$ we have that
\begin{equation*}
\int_{\mathbb{R}^{N}}\mathrm{e}^{-\mu_{1}\lvert x-x_{1}\rvert}\mathrm{e}
^{-\mu_{2}\lvert x-x_{2}\rvert}\,\mathrm{d}x\leq\int_{\mathbb{R}^{N}
}\mathrm{e}^{-\bar{\mu}\lvert x_{1}-x_{2}\rvert}\mathrm{e}^{-(\mu_{2}-\bar
{\mu})\lvert x-x_{2}\rvert}\,\mathrm{d}x=C\mathrm{e}^{-\bar{\mu}\lvert
x_{1}-x_{2}\rvert},
\end{equation*}
as claimed. The proof of the other inequality is similar.
\end{proof}

\begin{lemma}
\label{lem:splitting-f} There exists $\alpha\in(1/2,1]$ with the following
property: for any given $\widetilde{C}_{1}\geq1$ and $n\in\mathbb{N}$ there is
a constant $\widetilde{C}_{2}=\widetilde{C}_{2}(\alpha,n,\widetilde{C}_{1})>0$
such that the inequalities
\begin{equation}
\biggl \lvert f\biggl(\sum_{i=1}^{n}u_{i}\biggr)-\sum_{i=1}^{n}f(u_{i}
)\biggr \rvert\leq\widetilde{C}_{2}\sum_{i<j}\lvert u_{i}u_{j}\rvert^{\alpha},
\label{eq:21}
\end{equation}
\begin{equation}
\biggl \lvert F\biggl(\sum_{i=1}^{n}u_{i}\biggr)-\sum_{i=1}^{n}F(u_{i}
)-\sum_{i\neq j}f(u_{i})u_{j}\biggr \rvert\leq\widetilde{C}_{2}\biggl(\sum
_{i<j}\lvert u_{i}u_{j}\rvert^{2\alpha}+\sum_{i<j<k}\lvert u_{i}u_{j}
u_{k}\rvert^{2/3}\biggr), \label{eq:22}
\end{equation}
hold true for all $u_{1},u_{2},\dots,u_{n}\in\mathbb{R}$ with $\lvert
u_{i}\rvert\leq\widetilde{C}_{1}$.
\end{lemma}

\begin{proof}
Observe that (H3) implies that there is a constant $C>0$ such that
\begin{equation}
\left\vert f^{(k)}(u)\right\vert \leq C\left\vert u\right\vert ^{p_{1}
-k}\qquad\text{if }\left\vert u\right\vert \leq\widetilde{C}_{1},\ u\neq0.
\label{eq:29}
\end{equation}
Set $\alpha:=\min\{(p_{1}+1)/4,1\}\in(1/2,1]$. It is tedious but elementary to
prove that the inequalities
\begin{equation}
\left\vert f(u+v)-f(u)-f(v)\right\vert \leq C\left\vert uv\right\vert
^{\alpha} \label{eq:9}
\end{equation}
and
\begin{equation}
\left\vert F(u+v)-F(u)-F(v)-f(u)v-f(v)u\right\vert \leq C\left\vert
uv\right\vert ^{2\alpha} \label{eq:11}
\end{equation}
hold true for some constant $C>0$, if $\left\vert u\right\vert ,\left\vert
v\right\vert \leq\widetilde{C}_{1}$. These are inequalities \eqref{eq:21} and
\eqref{eq:22} for $n=2$. For $n>2$ inequalities \eqref{eq:21} and
\eqref{eq:22} follow easily by induction on $n$.
\end{proof}

The right-hand side of inequality \eqref{eq:22} indicates that we will need to
consider triple interactions. The following lemma will be useful to estimate them.

\begin{lemma}
\label{lem:triangle-geometry} Consider a triangle in $\mathbb{R}^{N}$ with
vertices $x_{1},x_{2},x_{3}\in\mathbb{R}^{N}$ and side lengths $w\leq v\leq
u$. Denote $s:=\min_{x\in\mathbb{R}^{N}}\sum_{k=1}^{3}\lvert x-x_{k}\rvert$.
Then the following statements are true:

\begin{enumerate}
\item[(a)] If one of the interior angles is larger than or equal to $2\pi/3$,
then $s=v+w$.

\item[(b)] In any case, $s\geq(w+v+u)/2$.
\end{enumerate}
\end{lemma}

\begin{proof}
The following facts from triangle geometry may be found in \cite{MR1573157}. The
minimum $s$ is achieved at a unique point $x_{0}$ in $\mathbb{R}^{N}$. In case
(a) that point is the vertex of the triangle with the largest interior angle,
so the claim follows immediately.

To prove (b) observe that adding up the inequalities $\lvert x_{i}-x_{0}
\rvert+\lvert x_{j}-x_{0}\rvert\geq\lvert x_{i}-x_{j}\rvert$, $i\neq j$,
yields
\begin{equation*}
2s=2\sum_{k=1}^{3}\lvert x_{0}-x_{k}\rvert\geq w+v+u\qquad\forall
x\in\mathbb{R}^{N},
\end{equation*}
as claimed.
\end{proof}

\begin{lemma}
\label{lem:est-intersection-balls} For $n\in\mathbb{N}$ there is a constant
$C=C(n)$ such that if $x_{1},x_{2}\in\mathbb{R}^{n}$ satisfy $\lvert
x_{1}-x_{2}\rvert<1$ and if $r\in\lbrack1,\lvert x_{1}-x_{2}\rvert+1]$ then
\begin{align}
\text{\emph{vol}}_{n}\left(  B_{r}(x_{2})\smallsetminus B_{1}(x_{1})\right)
&  \leq C\left(  \left\vert x_{1}-x_{2}\right\vert +r-1\right)  ,
\label{eq:53}\\begin{equation*}1ex]
\sup_{x\in\partial B_{r}(x_{2})}\text{\emph{dist}}(x,\partial B_{1}(x_{1}))
&  \leq\left\vert x_{1}-x_{2}\right\vert +r-1,\label{eq:35}\\
\sup_{x\in\partial B_{1}(x_{1})}\text{\emph{dist}}(x,\partial B_{r}(x_{2}))
&  \leq\left\vert x_{1}-x_{2}\right\vert +r-1. \label{eq:42}
\end{align}
Here $\emph{vol}_{n}$ denotes the Lebesgue measure in $\mathbb{R}^{n}$ and
$B_{r}(x):=\{y\in\mathbb{R}^{n}\mid\lvert y-x\rvert<r\}$.
\end{lemma}
\begin{proof}
Let $\omega_{k}$ denote the volume of the unit ball in $\mathbb{R}^{k}$. Set
$d:=\left\vert x_{1}-x_{2}\right\vert $. Without loss of generality we may
suppose that $x_{1}=0$ and $x_{2}=(d,0,\dots,0)$. Set $B_{1}:=B_{1}(0).$ Since
$B_{r}(x_{2})\smallsetminus B_{1}\subset(B_{1}(x_{2})\smallsetminus B_{1}
)\cup(B_{r}(x_{2})\smallsetminus B_{1}(x_{2}))$ and $r\in\lbrack1,2], $ we
have that
\begin{align*}
\text{vol}_{n}(B_{r}(x_{2})\smallsetminus B_{1})  &  \leq\text{vol}_{n}
(B_{1}(x_{2})\smallsetminus B_{1})+\omega_{n}(r^{n}-1)\\
&  \leq\text{vol}_{n}(B_{1}\smallsetminus B_{1}(x_{2}))+\omega_{n}
(2^{n}-1)(r-1).
\end{align*}
Write $x=(t,y)\in\mathbb{R}^{n}$ with $t\in\mathbb{R}$ and $y\in
\mathbb{R}^{n-1}.$ By symmetry considerations,
\begin{equation*}
\text{vol}_{n}(B_{1}\smallsetminus B_{1}(x_{2}))=\text{vol}_{n}\{x\in
B_{1}\mid\left\vert t\right\vert \leq d/2\}\leq\omega_{n-1}d.
\end{equation*}
Together with the previous inequality, this proves \eqref{eq:53}. An obvious
geometric argument proves \eqref{eq:35} and \eqref{eq:42}.
\end{proof}
\subsection{Analysis of linear operators and the limit problem}
\label{sec:analys-line-oper}Next we will show that $-\Delta+\lambda$ satisfies
the strong maximum principle on $\mathbb{L}$ and $\Omega_{R}$ for $R$ large if
$\lambda>-\lambda_{1,1}$.
For $r>0$ let $\lambda_{1,r}$ denote the smallest Dirichlet eigenvalue of
$-\Delta$ in the open ball $B_{r}^{N-1}:=\{\eta\in\mathbb{R}^{N-1}\mid
\lvert\eta\rvert<r\}$ of radius $r$ in $\mathbb{R}^{N-1}$, and let
$\vartheta_{1,r}$ be the positive eigenfunction corresponding to
$\lambda_{1,r}$, normalized by $\lVert\vartheta_{1,r}\rVert_{L^{2}}=1$. The
following result is well known.
\begin{lemma}
\label{lem:schroed-invertible} If $\lambda_{1}(\mathbb{L})$ denotes the bottom
of the spectrum of $-\Delta$ in $L^{2}(\mathbb{L})$ with Dirichlet boundary
conditions, then $\lambda_{1}(\mathbb{L})=\lambda_{1,1}$.
\end{lemma}
Next, we construct a positive superharmonic function for $-\Delta+\lambda$ in
$\Omega_{R}$ for $R$ large. This allows to estimate the bottom of the spectrum
of $-\Delta$ in $L^{2}(\Omega_{R})$ from below and provides a maximum
principle for $-\Delta+\lambda.$
As before, we write a point in $\mathbb{R}^{N}$ as $(\xi,\eta)$, where $\xi
\in\mathbb{R}$ and $\eta\in\mathbb{R}^{N-1}.$
\begin{lemma}
\label{lem:existence-positive-supersolution} If $\lambda>-\lambda_{1,1},$
there exists a superharmonic function for $-\Delta+\lambda$ in $C^{2}
(\mathbb{L})\cap C(\overline{\mathbb{L}})$ which is positive on $\overline
{\mathbb{L}}$. If $R$ is large enough then there exists a superharmonic
function for $-\Delta+\lambda$ in $C^{2}(\Omega_{R})\cap C(\overline
{\Omega_{R}})$ which is positive on $\overline{\Omega_{R}}$.
\end{lemma}
\begin{proof}
We fix $r>1$ close enough to $1$ so that $\lambda_{1,r}+\lambda>0$. Then
$W(\xi,\eta):=\vartheta_{1,r}(\eta)$ satisfies
\begin{equation*}
(-\Delta+\lambda)W=(\lambda_{1,r}+\lambda)W>0\text{ in }\mathbb{L}\text{
\ \ and \ \ }\min_{\overline{\mathbb{L}}}W>0.
\end{equation*}
This proves the first assertion.
To prove the second one note first that, for $R\geq1$ large enough, the set
$\Omega_{R,r}:=\{x\in\mathbb{R}^{N}\mid$dist$(x,\Gamma_{R})<r\}$ is a tubular
neighborhood of $\Gamma_{R}$. Since $\vartheta_{1,r}$ is radial, we may write
$\vartheta_{1,r}(\eta)=\vartheta_{1,r}(\left\vert \eta\right\vert )$ and
define
\begin{equation*}
W(x):=\vartheta_{1,r}(\text{dist}(x,\Gamma_{R}))\text{\quad for }x\in
\Omega_{R,r}.
\end{equation*}
Clearly, $\min\limits_{\overline{\Omega_{R}}}W>0$ for $R$ large enough. We
claim that
\begin{equation}
W\in C^{2}(\Omega_{R})\cap C(\overline{\Omega_{R}}) \label{eq:99}
\end{equation}
and
\begin{equation}
\min_{\Omega_{R}}\left(  {(-\Delta+\lambda)W}\right)  >0 \label{eq:100}
\end{equation}
for $R$ large enough. To prove this claims we fix $y_{0}\in\Omega_{R}$ and we
define locally, around $y_{0},$ a diffeomorphism from $\Omega_{R}$ to the unit
normal bundle of $\Gamma_{R}$ as follows: after a change of coordinates we may
assume that $0\in\Gamma_{R}$ and that dist$(y_{0},\Gamma_{R})=\left\vert
y_{0}\right\vert .$ We may also assume that the tangent space to $\Gamma_{R}$
at $0$ is $\mathbb{R}\times\{0\}.$ Then, $y_{0}\in\{0\}\times\mathbb{R}
^{N-1}.$ Let $\tau\colon(-\varepsilon,\varepsilon)\rightarrow\mathbb{R}^{N}$
be a parametrization by arc length of $\Gamma$ such that $\tau(0)=0$ and
$\tau^{\prime}(0)=(1,0).$ For $\xi\in(-R\varepsilon,R\varepsilon)$ and
$\eta\in\mathbb{R}^{N-1},$ set $\tau_{R}(\xi):=R\tau(\frac{\xi}{R})$ and let
$h_{R}(\xi,\eta)$ be the orthogonal projection of $(0,\eta)$ onto the space
$\tau^{\prime}(\xi)^{\perp}=\{x\in\mathbb{R}^{N}:x\cdot\tau^{\prime}(\frac
{\xi}{R})=0\}.$ Now define
\begin{equation*}
\Phi_{R}(\xi,\eta):=\tau_{R}\left(  \xi\right)  +\frac{\left\vert
\eta\right\vert }{\left\vert h_{R}(\xi,\eta)\right\vert }h_{R}(\xi,\eta).
\end{equation*}
Note that $\Phi_{R}(0,\eta)=(0,\eta)$. Moreover,
\begin{equation}
\mathrm{D}\Phi_{R}(0,\eta)=
\begin{pmatrix}
1-\frac{1}{R}\left[  (0,\eta)\cdot\tau^{\prime\prime}(0)\right]  & 0\\
0 & I_{N-1}
\end{pmatrix}
. \label{eq:87}
\end{equation}
Therefore, $\Phi_{R}$ is a $C^{2}$-diffeomorphism between neighborhoods of
$\overline{\{0\}\times B_{1}^{N-1}}$ for $R$ large enough. Note that, since
$h_{R}(\xi,\eta)$ is orthogonal to $\Gamma_{R}$ at $\tau_{R}(\xi),$
\begin{equation*}
\text{dist}(\Phi_{R}(\xi,\eta),\Gamma_{R})=\left\vert \Phi_{R}(\xi,\eta
)-\frac{1}{R}\tau(\xi)\right\vert =\left\vert \eta\right\vert .
\end{equation*}
This implies that
\begin{equation*}
W(\Phi_{R}(\xi,\eta))=\vartheta_{1,r}(\left\vert \eta\right\vert ).
\end{equation*}
So, since $\Phi_{R}$ is a local $C^{2}$-diffeomorphism at $y_{0}$, this
identity proves \eqref{eq:99}.
To prove \eqref{eq:100} it is enough to show that
\begin{equation}
(-\Delta+\lambda)W(0,\eta)\geq C>0 \label{eq:88}
\end{equation}
for $\eta\in B_{1}^{N-1}$ and large $R$, where $C$ is independent of $y_{0}$
and $R$. A straightforward computation shows that
\begin{equation*}
(-\Delta+\lambda)W(0,\eta)=\left(  {\lambda_{1,r}+\lambda}\right)
\vartheta_{1,r}(\eta)+O(\left\vert \mathrm{D}^{2}\Phi_{R}(0,\eta)\right\vert
),
\end{equation*}
independently of $y_{0}$, and that that
\begin{equation}
\mathrm{D}^{2}\Phi_{R}(0,\eta)\rightarrow0\qquad\text{as $R\rightarrow\infty$,
independently of $y_{0}$ and $\eta$.} \label{eq:102}
\end{equation}
Since $\vartheta_{1,r}$ is positive and continuous on $\overline{B_{1}^{N-1}}$
we may set
\begin{equation*}
C:=\frac{\lambda_{1,r}+\lambda}{2}\min_{\left\vert \eta\right\vert \leq
1}\vartheta_{1,r}(\eta)>0
\end{equation*}
and obtain \eqref{eq:88}.
\end{proof}
\begin{corollary}
\label{cor:spectrum-positive-omega-r} If $\lambda_{1}(\Omega_{R})$ denotes the
bottom of the spectrum of $-\Delta$ in $L^{2}(\Omega_{R})$ with Dirichlet
boundary conditions, then
\begin{equation*}
\liminf_{R\rightarrow\infty}\lambda_{1}(\Omega_{R})\geq\lambda_{1,1}.
\end{equation*}
\end{corollary}
\begin{proof}
A standard argument, using Lemma~\ref{lem:existence-positive-supersolution},
proves this claim.
\end{proof}
The following fact will play a crucial role to obtain asymptotic estimates for
the energy functional and its gradient.
\begin{corollary}
\label{rem:maximum-principle} If $\lambda>-\lambda_{1,1}$ the operator
$-\Delta+\lambda$ satisfies the strong maximum principle in any subdomain of
$\mathbb{L}$ and in any subdomain of $\Omega_{R}$ for $R$ large enough.
\end{corollary}
\begin{proof}
This follows from Lemma~\ref{lem:existence-positive-supersolution} and
\cite[Theorem~1]{MR0271508}.
\end{proof}
We shall also need the following decay estimates for the solutions $U^{\pm}$
to the limit problem \eqref{eq:2}. They follow immediately from
\cite[Proposition~4.2]{MR1039342}.
\begin{lemma}
\label{decayU} There are constants $C_{1},C_{2}>0$ such that
\begin{equation*}
C_{1}\mathrm{e}^{-\mu\lvert\xi\rvert}\vartheta_{1,1}(\eta)\leq\lvert U^{\pm
}(\xi,\eta)\rvert\leq C_{2}\mathrm{e}^{-\mu\lvert\xi\rvert}\vartheta
_{1,1}(\eta)\qquad\text{for all }(\xi,\eta)\in\mathbb{L}
\end{equation*}
where $\mu:=\sqrt{\lambda+\lambda_{1,1}}.$
\end{lemma}
\section{Asymptotics of the energy and its gradient}
\label{sec:energy-estimates}
We assume from now on that $\lambda>-\lambda_{1,1}.$ Let $\mathbb{L}
_{s}:=\{(\xi,\eta)\in\mathbb{R}^{1}\times\mathbb{R}^{N-1}:\lvert\eta
\rvert<s\}.$ We fix $r_{0}>1$ such that $\lambda_{1,r_{0}}+\lambda>0$ and, for
$R>0$, $x\in\Gamma_{R}$ and $s\in\lbrack1,r_{0}],$ we set
\begin{equation*}
\mathbb{L}_{s,x}:=\{x+A_{x}^{-1}(z):z\in\mathbb{L}_{s}\}
\end{equation*}
with $A_{x}$ as in (\ref{eq:3})$.$ Note that the first eigenvalue of $-\Delta$
in $H_{0}^{1}(\Omega_{R}\cap\mathbb{L}_{s,x})$ satisfies $\lambda_{1}
(\Omega_{R}\cap\mathbb{L}_{s,x})+\lambda>0$ for large $R$, because $\Omega
_{R}\cap\mathbb{L}_{s,x}$ is an open bounded subset of $\mathbb{L}_{r_{0},x}$.
We write $V_{x,s,R}^{\pm}$ for the unique solution to the problem
\begin{equation}
\left\{
\begin{array}
[c]{ll}
-\Delta u+\lambda u=f(U_{x,R}^{\pm}) & \text{in }\Omega_{R}\cap\mathbb{L}
_{s,x},\\
u=0 & \text{on }\partial\left(  \Omega_{R}\cap\mathbb{L}_{s,x}\right)  ,
\end{array}
\right.  \label{V}
\end{equation}
with $U_{x,R}^{\pm}$ as in (\ref{eq:3}). By the maximum principle and
assuption (H4), $V_{x,s,R}^{+}$ is positive and $V_{x,s,R}^{-}$ is negative
for large $R.$ We extend $V_{x,s,R}^{\pm}$ to all of $\mathbb{R}^{N}$ by
defining it as $0$ outside of $\Omega_{R}\cap\mathbb{L}_{s,x}$.\ When $s=1$ we
omit it from the notation and write $\mathbb{L}_{x},$ $V_{x,R}$ instead of
$\mathbb{L}_{1,x},$ $V_{x,1,R}.$
\subsection{The closed tube case}
In this subsection we assume that $\gamma(0)=\gamma(1).$ The following decay
estimates hold true.
\begin{lemma}
\label{lem:decay-V}For each $s\in\lbrack1,r_{0})$ there are positive constants
$c_{3},c_{4}$ and $R_{0},$ independent of $x\in\Gamma_{R},$ such that all
quantities
\begin{equation*}
\left\vert U_{x,R}^{\pm}(y)\right\vert ,\quad\left\vert \nabla U_{x,R}^{\pm
}(y)\right\vert ,\quad\left\vert V_{x,s,R}^{\pm}(y)\right\vert ,\quad
\left\vert \nabla V_{x,s,R}^{\pm}(y)\right\vert ,
\end{equation*}
are bounded by $c_{3}e^{-c_{4}\left\vert y-x\right\vert }$ for all $R\geq
R_{0}$ and almost all $y\in\mathbb{R}^{N}.$ Moreover,
\begin{equation*}
\left\vert D^{2}U_{x,R}^{\pm}(y)\right\vert \text{\quad and\quad}\left\vert
D^{2}V_{x,s,R}^{\pm}(y)\right\vert
\end{equation*}
are bounded uniformly in $\mathbb{L}_{x}$\ and $\Omega_{R}\cap\mathbb{L}
_{s,x}$\ respectively, independently of $R\geq R_{0}.$
\end{lemma}
\begin{proof}
Lemma \ref{decayU}, together with standard regularity estimates, yields the
estimates for $U_{x,R}^{\pm}$ and its derivatives.\newline To prove the
estimates for $V_{x,s,R}^{\pm}$ we assume without loss of generality that
$x=0$ and that $\mathbb{R}\times\{0\}$ is the tangent space to $\Gamma_{R}$ at
$0.$ Then there exists $\tilde{c}_{s}>0$ such that $\vartheta_{1,r_{0}}
(\eta)\geq\tilde{c}_{s}$ for all $\eta\in B_{s}^{N-1}, $ where $\vartheta
_{1,r_{0}}$ is the positive first Dirichlet eigenfunction of $-\Delta$ in the
ball of radius $r_{0}$ (as in the beginning of subsection
\ref{sec:analys-line-oper}). We write $y\in\mathbb{L}_{s}$ as $(\xi,\eta)$
with $\xi\in\mathbb{R}$ and $\eta\in B_{s}^{N-1}, $ and set
\begin{equation*}
W(y):=e^{-\nu\left\vert \xi\right\vert }\vartheta_{1,r_{0}}(\eta)
\end{equation*}
where $\nu$ is a small positive constant, independent of $R,$ which will be
fixed next. A straightforward computation gives
\begin{align*}
-\Delta W(y)+\lambda W(y)  &  =\left(  \frac{\left(  N-1\right)  \nu
}{\left\vert \xi\right\vert }-\nu^{2}+\lambda_{1,r_{0}}+\lambda\right)  W(y)\\
&  >\left(  \lambda_{1,r_{0}}+\lambda-\nu^{2}\right)  \tilde{c}_{s}
e^{-\nu\left\vert \xi\right\vert }.
\end{align*}
Since $\lambda_{1,r_{0}}+\lambda>0$ we have that $\lambda_{1,r_{0}}
+\lambda-\nu^{2}>0$ if $\nu$ is small enough. On the other hand, assumption
(H3) on $f$ together with Lemma \ref{decayU} yield that
\begin{equation*}
f(U_{x,R}^{+})\leq b_{1}e^{-\mu p_{1}\left\vert \xi\right\vert },
\end{equation*}
for some large enough $b_{1}>0$. Since $V_{x,s,R}^{+}$ satisfies (\ref{V}) the
maximum principle implies that $V_{x,s,R}^{+}\leq b_{2}W$ with $b_{2}
:=b_{1}\tilde{c}_{s}^{-1}\left(  \lambda_{1,r_{0}}+\lambda-\nu^{2}\right)
^{-1}.$ This gives the exponential bound on $V_{x,s,R}^{+}.$ Similarly for
$V_{x,s,R}^{-}.$ Regularity estimates, using the results in \cite{MR1156467}, yield
the estimates for its derivatives.
\end{proof}
Set
\begin{equation*}
F(u):=\int_{0}^{u}f(s)\,\mathrm{d}s\qquad\text{if }u\in\mathbb{R}\text{.}
\end{equation*}
Then, by (H3),
\begin{equation}
\lvert F(u)\rvert\leq C(\lvert u\rvert^{p_{1}+1}+\lvert u\rvert^{p_{2}
+1})\qquad\text{for all }u\in\mathbb{R}\text{.} \label{eq:49}
\end{equation}
\begin{lemma}
\label{lem:u-v-est}For $s\in\lbrack1,r_{0})$ and $p\in(0,\infty)$ the
asymptotic estimates
\begin{align}
\int_{\mathbb{R}^{N}}\lvert V_{x,s,R}^{\pm}-U_{x,R}^{\pm}\rvert^{p}  &
=O(R^{-\min\{p,1\}}),\label{eq:13}\\
\int_{\mathbb{R}^{N}}\lvert\nabla V_{x,s,R}^{\pm}-\nabla U_{x,R}^{\pm}
\rvert^{2}  &  =O(R^{-1}),\label{eq:12}\\
\int_{\mathbb{R}^{N}}\lvert F(V_{x,s,R}^{\pm})-F(U_{x,R}^{\pm})\rvert &
=O(R^{-1}),\label{eq:14}\\
\int_{\mathbb{R}^{N}}\lvert f(V_{x,s,R}^{\pm})-f(U_{x,R}^{\pm})\rvert^{p}  &
=O(R^{-\min\{p,1\}}), \label{eq:15}
\end{align}
hold true as $R\rightarrow\infty,$ independently of $x\in\Gamma_{R}.$
\end{lemma}
\begin{proof}
Let $x$ be a point on $\Gamma.$ After translation and rotation we may assume
that $x=0$ and that $\mathbb{R}\times\{0\}$ is the tangent space to $\Gamma$
at $0.$ Since $\Gamma$ is compact there exist $\delta,\rho>0$, independent of
$x$, and a $C^{3}$-function $h:(-\rho,\rho)\rightarrow B_{\delta}^{N-1}$ such
that
\begin{equation*}
\Gamma\cap\left(  (-\rho,\rho)\times B_{\delta}^{N-1}\right)  =\{(\xi
,h(\xi))\mid\xi\in(-\rho,\rho)\},
\end{equation*}
and the derivatives of $h$ up to the order $3$ are bounded independently of
$\xi\in(-\rho,\rho)$ and $x\in\Gamma.$ Setting $h_{R}(\xi):=Rh(\xi/R)$ we have
that
\begin{equation*}
\widetilde{\Gamma}_{R}:=\Gamma_{R}\cap\left(  (-\rho R,\rho R)\times B_{\delta
R}^{N-1}\right)  =\{(\xi,h_{R}(\xi))\mid\xi\in(-\rho R,\rho R)\}.
\end{equation*}
An easy argument using Taylor's theorem and geometric considerations shows
that there is a constant $C,$ independent of $x$, such that
\begin{equation}
\left\vert h_{R}(\xi)\right\vert \leq\frac{C\xi^{2}}{R},\text{\quad}\left\vert
h_{R}^{\prime}(\xi)\right\vert \leq\frac{C\left\vert \xi\right\vert }
{R}\text{\quad and\quad}\left\vert y-h_{R}(\xi)\right\vert \leq1+\frac
{C(1+\xi^{2})}{R^{2}} \label{eq:55}
\end{equation}
for all $\xi\in(-\rho R+1,\rho R-1)$ and $y\in\mathbb{R}^{N-1}$ with
$(\xi,y)\in\Omega_{R}$. It follows that
\begin{equation}
\{\xi\}\times B_{1}^{N-1}(h_{R}(\xi))\subset\left[  \{\xi\}\times
\mathbb{R}^{N-1}\right]  \cap{\Omega}_{R}\subset\{\xi\}\times
B_{1+C(1+\left\vert \xi\right\vert ^{2})/R^{2}}^{N-1}(h_{R}(\xi))
\label{eq:56}
\end{equation}
for all $\xi\in(-\rho R+1,\rho R-1)$ and $R$ large enough. Consider the set
\begin{equation*}
Q_{R}:=(-R^{1/4},R^{1/4})\times B_{s}^{N-1}\subset\mathbb{L}_{s}.
\end{equation*}
We express $\mathbb{R}^{N}$ as the union of the sets
\begin{equation}
\mathbb{R}^{N}\smallsetminus Q_{R},\hspace{0.25in}Q_{R}\cap\left(  \Omega
_{R}\smallsetminus\mathbb{L}\right)  ,\hspace{0.25in}Q_{R}\cap\left(
\mathbb{L}\smallsetminus\Omega_{R}\right)  ,\hspace{0.25in}Q_{R}\cap
\mathbb{L}\cap\Omega_{R}. \label{eq:63}
\end{equation}
We will show that the estimates \eqref{eq:13}, \eqref{eq:12}, \eqref{eq:14},
\eqref{eq:15}, hold true for the integrals over each one of these sets. Note
that the integrals over $Q_{R}\smallsetminus\left(  \mathbb{L}\cup\Omega
_{R}\right)  $ are zero.
\begin{claim}
\label{c1}Estimate \eqref{eq:13} holds true for the integral over
$\mathbb{R}^{N}\smallsetminus Q_{R}.$
\end{claim}
By Lemma~\ref{lem:decay-V} there are positive constants $\widetilde{C}
_{1},\widetilde{C}_{2}$ such that
\begin{equation}
\lvert V_{x,s,R}^{\pm}(\xi,\eta)-U_{x,R}^{\pm}(\xi,\eta)\rvert^{p}
\leq\widetilde{C}_{1}\mathrm{e}^{-\widetilde{C}_{2}(\left\vert \xi\right\vert
+\left\vert \eta\right\vert )} \label{eq:59}
\end{equation}
for all $(\xi,\eta)\in\mathbb{R}\times\mathbb{R}^{N-1}$. This immediately
yields Claim~\ref{c1}.
\begin{claim}
\label{c2}Estimate \eqref{eq:13} holds true for the integral over $Q_{R}
\cap\left(  \Omega_{R}\smallsetminus\mathbb{L}\right)  .$
\end{claim}
By \eqref{eq:59}, \eqref{eq:56}, Lemma~\ref{lem:est-intersection-balls} and
\eqref{eq:55} it holds that
\begin{align*}
&  \int_{Q_{R}\cap\left(  \Omega_{R}\smallsetminus\mathbb{L}\right)  }\lvert
V_{x,s,R}^{\pm}-U_{x,R}^{\pm}\rvert^{p}\text{ }\\
&  \leq\widetilde{C}_{1}\int_{-R^{1/4}}^{R^{1/4}}\mathrm{e}^{-\widetilde
{C}_{2}\left\vert \xi\right\vert }\text{vol}_{N-1}\left(  B_{1+C(1+\left\vert
\xi\right\vert ^{2})/R^{2}}^{N-1}(h_{R}(\xi))\smallsetminus B_{1}
^{N-1}(0)\right)  \mathrm{d}\xi\\
&  \leq C\int_{-R^{1/4}}^{R^{1/4}}\mathrm{e}^{-\widetilde{C}_{2}\left\vert
\xi\right\vert }(\left\vert h_{R}(\xi)\right\vert +C(1+\xi^{2})/R^{2})\text{
}\mathrm{d}\xi\\
&  \leq\frac{C}{R}\int_{-\infty}^{\infty}\mathrm{e}^{-\widetilde{C}
_{2}\left\vert \xi\right\vert }(1+\xi^{2})\text{ }\mathrm{d}\xi=O(R^{-1})
\end{align*}
as $R\rightarrow\infty.$
\begin{claim}
\label{c3}Estimate \eqref{eq:13} holds true for the integral over $Q_{R}
\cap\left(  \mathbb{L}\smallsetminus\Omega_{R}\right)  .$
\end{claim}
The proof is similar to that of Claim 2, using this time the first inclusion
in \eqref{eq:56}.
\begin{claim}
\label{c4}Estimate \eqref{eq:13} holds true for the integral over $Q_{R}
\cap\mathbb{L}\cap\Omega_{R}.$
\end{claim}
Set $D_{R}:=Q_{R}\cap\mathbb{L}\cap\Omega_{R}.$ First we prove that, for some
suitable constant $C$ independent of $x$ and $R$,
\begin{equation}
\left\vert V_{x,s,R}^{\pm}(\xi,\eta)-U_{x,R}^{\pm}(\xi,\eta)\right\vert \leq
C\mathrm{e}^{-C_{4}\left\vert \xi\right\vert }\frac{1+\xi^{2}}{R}
\label{eq:64}
\end{equation}
for all $(\xi,\eta)\in\partial D_{R}$. Let $(\xi,\eta)\in\partial D_{R}.$ If
$(\xi,\eta)\in\partial\mathbb{L}$, Lemma~\ref{lem:est-intersection-balls},
together with \eqref{eq:56}, and \eqref{eq:55}, yields
\begin{equation}
\text{dist}((\xi,\eta),\partial\Omega_{R})\leq\left\vert h_{R}(\xi)\right\vert
+C\frac{1+\xi^{2}}{R^{2}}\leq C\frac{1+\xi^{2}}{R}. \label{eq:85}
\end{equation}
Similarly, if $(\xi,\eta)\in\partial\Omega_{R}$ then
\begin{equation*}
\text{dist}((\xi,\eta),\partial\mathbb{L})\leq C\frac{1+\xi^{2}}{R}.
\end{equation*}
Since $U_{x,R}^{\pm}$ vanishes on $\partial\mathbb{L}$ and $V_{x,s,R}^{\pm}$
vanishes on $\partial\Omega_{R}$, the estimates in Lemma~\ref{lem:decay-V}
yield inequality \eqref{eq:64}. Next we set $W(y):=e^{-\nu\left\vert
\xi\right\vert }\vartheta_{1,r_{0}}(\eta)$ with $\nu\in(0,C_{4})$ as in the
proof of Lemma~\ref{lem:decay-V}. By \eqref{eq:64} there exists $C>0$ such
that
\begin{equation*}
\left\vert V_{x,s,R}^{\pm}(\xi,\eta)-U_{x,R}^{\pm}(\xi,\eta)\right\vert
\leq\frac{C}{R}W(\xi,\eta)
\end{equation*}
for all $(\xi,\eta)\in\partial D_{R}$. Since $V_{x,R}^{\pm}-U_{x,R}^{\pm}$ is
harmonic for $-\Delta+\lambda$ in $D_{R}$ the maximum principle implies that
\begin{equation*}
\left\vert V_{x,s,R}^{\pm}(\xi,\eta)-U_{x,R}^{\pm}(\xi,\eta)\right\vert
\leq\frac{C}{R}W(\xi,\eta)=\frac{C}{R}e^{-\nu\left\vert \xi\right\vert
}\vartheta_{1,r_{0}}(\eta)
\end{equation*}
for all $(\xi,\eta)\in D_{R}$, with $C$ independent of $x$ and $R$.
Therefore,
\begin{equation}
\int_{D_{R}}\left\vert V_{x,s,R}^{\pm}-U_{x,R}^{\pm}\right\vert ^{p}=O(R^{-p})
\label{eq:67}
\end{equation}
as $R\rightarrow\infty$. This proves Claim \ref{c4}.
\begin{claim}
\label{c5}Estimate \eqref{eq:12} holds true for the integrals over
$\mathbb{R}^{N}\smallsetminus Q_{R},$ $Q_{R}\cap\left(  \Omega_{R}
\smallsetminus\mathbb{L}\right)  $ and $Q_{R}\cap\left(  \mathbb{L}
\smallsetminus\Omega_{R}\right)  .$
\end{claim}
The same arguments as in the proofs of Claims \ref{c1}, \ref{c2} and \ref{c3}
yield this claim.
\begin{claim}
\label{c6}Estimate \eqref{eq:12} holds true for the integral over $Q_{R}
\cap\mathbb{L}\cap\Omega_{R}.$
\end{claim}
Set $D_{R}:=Q_{R}\cap\mathbb{L}\cap\Omega_{R}.$ The functions $U_{x,R}^{\pm}$
and $V_{x,R}^{\pm}$ can be extended to $C^{2}$-functions in neighborhoods of
$\mathbb{L}$ and $\Omega_{R}$, respectively. Denote by $Y_{R}$ the difference
of these extensions on a neighborhood of $\overline{D_{R}}$. Note that $D_{R}$
has Lipschitz boundary if $R$ is large enough. Hence we can apply the
Gauss-Green theorem (see e.g.\ \cite[Theorem~5.8.2]{MR1014685} and the remark
following it) and obtain that
\begin{equation}
\int_{D_{R}}\left\vert \nabla Y_{R}\right\vert ^{2}=\int_{\partial D_{R}}
Y_{R}\,n_{R}(x)\cdot\nabla Y_{R}\,\mathrm{d}H_{N-1}(x)-\lambda\int_{D_{R}
}Y_{R}^{2} \label{eq:105}
\end{equation}
Here $n_{R}(x)$ denotes the measure theoretic exterior normal to $\partial
D_{R}$ at $x$, and $H_{N-1}$ denotes $(N-1)$-dimensional Hausdorff measure. By
Lemma~\ref{lem:decay-V}, $\nabla Y_{R}$ is bounded uniformly and independently
of $R$. Hence \eqref{eq:105} and \eqref{eq:64} imply
\begin{align*}
\int_{D_{R}}\left\vert \nabla Y_{R}\right\vert ^{2}  &  \leq\frac{C}{R}\left(
\int_{\partial D_{R}}\mathrm{e}^{-C_{4}\left\vert \xi\right\vert
}(1+\left\vert \xi\right\vert ^{2})\,\mathrm{d}H_{N-1}(x)+\int_{D_{R}
}\mathrm{e}^{-C_{4}\left\vert \xi\right\vert }(1+\left\vert \xi\right\vert
^{2})\,\mathrm{d}\xi\,\mathrm{d}\eta\right) \\
&  =O(R^{-1}).
\end{align*}
This proves Claim \ref{c6}.
\begin{claim}
\label{c7}Estimates \eqref{eq:14} and \eqref{eq:15} hold true.
\end{claim}
These estimates follow easily from \eqref{eq:13} since $U_{x,R}^{\pm}$ and
$V_{x,R}^{\pm}$ are bounded uniformly as $R\rightarrow\infty$ and $F$ and $f$
are continuously differentiable.
\end{proof}
The energy functional for the Dirichlet problem $-\Delta u+\lambda u=f(u)$ in
a domain $\Omega\subseteq\mathbb{R}^{N}$ is given by
\begin{equation*}
J_{\Omega}(u):=\frac{1}{2}\int_{\Omega}(\lvert\nabla u\rvert^{2}+\lambda
u^{2})-\int_{\Omega}F(u),\qquad u\in H_{0}^{1}(\Omega).
\end{equation*}
By (H2), (H3) and \eqref{eq:49} $J_{\Omega}$ is well defined and twice
continuously differentiable on $H_{0}^{1}(\Omega)$, with $\mathrm{D}
^{2}J_{\Omega}$ globally H\"{o}lder continuous on bounded subsets of
$H_{0}^{1}(\Omega)$.
\begin{lemma}
\label{lem:nearness-u-v} The estimates
\begin{align}
\sup_{x\in\Gamma_{R}}\lVert V_{x,R}^{\pm}-U_{x,R}^{\pm}\rVert_{H^{1}
(\mathbb{R}^{N})}  &  =O(R^{-1/2}),\label{eq:16}\\
\sup_{x\in\Gamma_{R}}\lvert J_{\Omega_{R}}(V_{x,R}^{\pm})-J_{\mathbb{L}
}(U^{\pm})\rvert &  =O(R^{-1}),\label{eq:17}\\
\sup_{x\in\Gamma_{R}}\lVert\nabla J_{\Omega_{R}}(V_{x,R}^{\pm})\rVert
_{H_{0}^{1}(\Omega_{R})}  &  =O(R^{-1/2}), \label{eq:18}
\end{align}
hold true as $R\rightarrow\infty$.
\end{lemma}
\begin{proof}
Estimates \eqref{eq:16} and \eqref{eq:17} follow immediately from Lemma
\ref{lem:u-v-est}. To prove the third one we choose $s\in(1,r_{0})$ and a
cut-off function $\chi\in C^{\infty}(\mathbb{R}^{N-1})$ with $\chi(\eta)=1$ if
$\left\vert \eta\right\vert \leq1$ and $\chi(\eta)=0$ if $\left\vert
\eta\right\vert \geq s.$ Fix $R$ and $x\in\Gamma_{R}.$ Assuming that $x=0$ and
that $\mathbb{R}\times\{0\}$ is the tangent space to $\Gamma_{R}$ at $0,$ we
write $v\in H_{0}^{1}(\Omega_{R})$ as $v=v_{1}+v_{2}$ where $v_{1}(\xi
,\eta):=\chi(\eta)v(\xi,\eta).$ Then $v_{1}\in H_{0}^{1}(\Omega_{R}
\cap\mathbb{L}_{s,x}),$ supp$(v_{2})\subset\Omega_{R}\smallsetminus
\mathbb{L}_{x}$ and there exists a constant $c_{s}, $ independent of $R$ and
$x,$ such that $\left\Vert v_{1}\right\Vert _{H^{1}(\mathbb{R}^{N})}\leq
c_{s}\left\Vert v\right\Vert _{H^{1}(\mathbb{R}^{N})}$ for all $v\in H_{0}
^{1}(\Omega_{R}).$ From the definition of $V_{x,s,R}^{\pm}$ and Lemma
\ref{lem:u-v-est} we obtain
\begin{align*}
\left\vert DJ_{R}(V_{x,R}^{\pm})v\right\vert  &  =\left\vert DJ_{R}
(V_{x,R}^{\pm})v_{1}\right\vert \\
&  \leq\left\vert DJ_{R}(V_{x,s,R}^{\pm})v_{1}\right\vert +\left\vert
DJ_{R}(V_{x,R}^{\pm})v_{1}-DJ_{R}(V_{x,s,R}^{\pm})v_{1}\right\vert \\
&  \leq\left\vert \int_{\mathbb{R}^{N}}\left(  f(U_{x,R}^{\pm})-f(V_{x,s,R}
^{\pm})\right)  v_{1}\right\vert +O(R^{-1/2})\left\Vert v_{1}\right\Vert
_{H^{1}(\mathbb{R}^{N})}\\
&  \leq O(R^{-1/2})\left\Vert v\right\Vert _{H^{1}(\mathbb{R}^{N})},
\end{align*}
as claimed.
\end{proof}
For $m=1,2$ we consider functions $g_{m}\colon\mathbb{R}^{+}\rightarrow
\mathbb{R}^{+}$ (to be fixed later) satisfying
\begin{align}
g_{2}  &  <g_{1}, &  & \label{eq:24}\\
g_{m}(R)  &  \rightarrow\infty & \text{as }R  &  \rightarrow\infty,\text{ for
}m=1,2,\label{eq:10}\\
g_{m}(R)  &  =o(R) & \text{as }R  &  \rightarrow\infty,\text{ for }m=1,2.
\label{eq:25}
\end{align}
Let $D_{m,R}$ be the set of points $(x_{1},x_{2},\dots,x_{n})$ in $(\Gamma
_{R})^{n}$ such that there exist $i,j\in\{1,2,\dots,n\}$ with $i\neq j$ and
$\lvert x_{i}-x_{j}\rvert\leq g_{m}(R)$, and let
\begin{equation}
\mathcal{U}_{m,R}:=\{(x_{1},x_{2},\dots,x_{n})\in(\Gamma_{R})^{n}
\smallsetminus D_{m,R}\mid(x_{1},x_{2},\dots,x_{n})\text{ is an $n $-chain}\},
\label{Uclosed}
\end{equation}
see \eqref{eq:33} for the definition of an $n$-chain. Then $\mathcal{U}_{1,R}
$ and $\mathcal{U}_{2,R}$ are open subsets of $(\Gamma_{R})^{n}$ such that
$\overline{\mathcal{U}_{1,R}}\subset\mathcal{U}_{2,R}$. For $i,j\in
\{1,2,\dots,n\}\ $we set
\begin{equation*}
d_{n}(i,j):=\min\{\left\vert i-j\right\vert ,\left\vert i-j+n\right\vert
,\left\vert i-j-n\right\vert \}.
\end{equation*}
$d_{n}(i,j)$ is the distance from $i$ to the set of integers which are
congruent to $j$ mod $n.$
\begin{lemma}
\label{lem:chains}For $R$ large enough and every $(x_{1},x_{2},\dots,x_{n}
)\in\overline{\mathcal{U}_{1,R}}$ we have that
\begin{equation}
s(R):=\min_{x\in\mathbb{R}^{N}}\left(  \lvert x-x_{i}\rvert+\lvert
x-x_{j}\rvert+\lvert x-x_{\ell}\rvert\right)  \geq2g_{1}(R) \label{sR}
\end{equation}
for any $i,j,\ell\in\{1,2,\dots,n\},$ and
\begin{equation}
\lvert x_{i}-x_{j}\rvert\geq\frac{4}{3}g_{1}(R)\text{\qquad if }d_{n}
(i,j)\geq2. \label{d>2}
\end{equation}
\end{lemma}
\begin{proof}
Since $\Gamma$ is compact there exists $\varrho>0$ with the following properties:
\begin{enumerate}
\item[(i)] If $x,y\in\Gamma$ and $0<\left\vert x-y\right\vert <2\varrho$ then
there exists a connected component $\mathcal{C}$ of  $\Gamma\smallsetminus
\{x,y\}$ such that $\left\vert x-z\right\vert +\left\vert z-y\right\vert
\leq\frac{3} {2}\left\vert x-y\right\vert $ for every $z\in\mathcal{C}$.
\item[(ii)] If $x,y,z$ are three different points in $\Gamma$, $\left\vert
x-y\right\vert <2\varrho$ and $\left\vert z-y\right\vert <2\varrho$ then one
of the angles of the triangle with vertices $x,y,z$ is larger that $2\pi/3.$
\end{enumerate}
Fix $R$ large enough so that $\frac{g_{1}(R)}{R}<\varrho.$
Let $(x_{1},x_{2},\dots,x_{n})\in\overline{\mathcal{U}_{1,R}}.$ If
\begin{equation*}
  \max\left(  \lvert x_{j}-x_{i}\rvert+\lvert x_{\ell}-x_{j}\rvert+\lvert
    x_{i}-x_{\ell}\rvert\right)  <2\varrho R,
\end{equation*}
the points $\frac{x_{i}}{R}
,\frac{x_{j}}{R},\frac{x_{\ell}}{R}\in\Gamma$ satisfy the hypothesis of (ii)
and, therefore, the triangle with vertices $x_{i},x_{j},x_{\ell}$ has an angle
which is larger that $2\pi/3.$ It follows from Lemma
\ref{lem:triangle-geometry} that $s(R)\geq2g_{1}(R).$ If, on the other hand,
$\lvert x_{j}-x_{i}\rvert\geq2\varrho R$ then $\lvert x_{j}-x_{i}\rvert
\geq2g_{1}(R),$ and Lemma \ref{lem:triangle-geometry} implies that
$s(R)\geq2g_{1}(R).$ This proves (\ref{sR}).
To prove (\ref{d>2}) we argue by contradiction. Assume there are $(x_{1}
,x_{2},\dots,x_{n})\in\overline{\mathcal{U}_{1,R}}$ and\ $i,j\in
\{1,2,\dots,n\}$ such that $d_{n}(i,j)\geq2$ and $\lvert x_{i}-x_{j}
\rvert<\frac{4}{3}g_{1}(R).$ Then $0<\left\vert \frac{x_{i}}{R}-\frac{x_{j}
}{R}\right\vert <2\varrho$. Since $d_{n}(i,j)\geq2$ there is a point $x_{\ell}
$ in the $n$-chain, which lies between $x_{i}$ and $x_{j},$ such that
$\frac{x_{\ell}}{R}$ belongs to the connected component of $\Gamma
\smallsetminus\{\frac{x_{i}}{R},\frac{x_{j}}{R}\}$ to which the conclusion of
(i) applies. Then, $2g_{1}(R)\leq\lvert x_{i}-x_{\ell}\rvert+\lvert x_{\ell
}-x_{j}\rvert\leq\frac{3}{2}\lvert x_{i}-x_{j}\rvert,$ a contradiction.
\end{proof}
In the rest of this subsection we assume that $n=2k.$ For $X=(x_{1}
,x_{2},\dots,x_{n})\in\mathcal{U}_{2,R}$ we define $\varphi_{R}\colon
\mathcal{U}_{2,R}\rightarrow H_{0}^{1}(\Omega_{R})$ by
\begin{equation}
\varphi_{R}(X):=
{\sum_{i=1}^{k}}
(V_{x_{2i-1},R}^{+}+V_{x_{2i},R}^{-}). \label{immersion}
\end{equation}
For fixed $X=(x_{1},x_{2},\dots,x_{n})$ it will be convenient to write
\begin{equation}
\overline{U}_{i}:=\left\{
\begin{array}
[c]{ll}
U_{x_{i},R}^{+} & \text{if }i\text{ is odd,}\\
U_{x_{i},R}^{-} & \text{if }i\text{ is even,}
\end{array}
\right.  \text{\qquad}\overline{V}_{i}:=\left\{
\begin{array}
[c]{ll}
V_{x_{i},R}^{+} & \text{if }i\text{ is odd,}\\
V_{x_{i},R}^{-} & \text{if }i\text{ is even.}
\end{array}
\right.  \label{bar}
\end{equation}
Then
\begin{equation*}
\varphi_{R}(X)=
{\sum_{i=1}^{n}}
\overline{V}_{i}.
\end{equation*}
\begin{proposition}
\label{prop:lyapunov-gradient-estimate} Let $\alpha$ be as in
\emph{Lemma~\ref{lem:splitting-f}} and fix $\alpha^{\prime}\in(1/2,\alpha)$.
Then
\begin{equation*}
\sup_{X\in\mathcal{U}_{2,R}}\lVert\nabla J_{\Omega_{R}}(\varphi_{R}
(X))\rVert_{H_{0}^{1}(\Omega_{R})}=O(\mathrm{e}^{-\alpha^{\prime}\mu g_{2}
(R)})+O(R^{-1/2})
\end{equation*}
as $R\rightarrow\infty$.
\end{proposition}
\begin{proof}
Fix $X=(x_{1},x_{2},\dots,x_{n})\in\mathcal{U}_{2,R}.$ If $v\in H_{0}
^{1}(\Omega_{R})$ satisfies $\left\Vert v\right\Vert _{H_{0}^{1}(\Omega_{R}
)}=1$ then, using Lemmas \ref{lem:nearness-u-v}, \ref{lem:splitting-f},
\ref{lem:u-v-est}, \ref{lem:decay-V} and \ref{lem:interaction-exponential}\ we
obtain
\begin{multline*}
\left\vert DJ_{\Omega_{R}}(\varphi_{R}(X))\left[  v\right]
\right\vert  \\
\begin{aligned}
  &=\left\vert {\sum_{i=1}^{n}}
    DJ_{\Omega_{R}}(\varphi_{R}(\overline{V}_{i}))\left[ v\right]
    +\int _{\Omega_{R}}\left( {\sum_{i=1}^{n}}
      f(\overline{V}_{i})-f\left( {\sum_{i=1}^{n}}
        \overline{V}_{i}\right)  \right)  v\right\vert \\
  & \leq {\sum_{i=1}^{n}} \lVert\nabla
  J_{\Omega_{R}}(\overline{V}_{i})\rVert_{H_{0}^{1}(\Omega_{R}
    )}+\left( \int_{\Omega_{R}}\left\vert {\sum_{i=1}^{n}}
      f(\overline{V}_{i})-f\left( {\sum_{i=1}^{n}}
        \overline{V}_{i}\right)  \right\vert ^{2}\right)  ^{1/2}\\
  & \leq O(R^{-1/2})+C {\sum_{i<j}} \left(
    \int_{\Omega_{R}}\lvert\overline{V}_{i}\overline{V}_{j}\rvert
    ^{2\alpha}\right)  ^{1/2}\\
  & =O(R^{-1/2})+C {\sum_{i<j}} \left(
    \int_{\Omega_{R}}\lvert\overline{U}_{i}\overline{U}_{j}\rvert
    ^{2\alpha}\right)  ^{1/2}\\
  & =O(R^{-1/2})+O(e^{-\alpha^{\prime}\mu g_{2}(R)}).
\end{aligned}
\end{multline*}
These estimates are independent of the choice of $X.$
\end{proof}
Recall that $n=2k$ and set
\begin{equation*}
E_{n}:=k\left[  J_{\mathbb{L}}(U^{+})+J_{\mathbb{L}}(U^{-})\right]  .
\end{equation*}
\begin{proposition}
\label{prop:jr-high-on-boundary} There exists $\beta>0$ such that
\begin{equation*}
\inf_{X\in\partial\mathcal{U}_{1,R}}J_{\Omega_{R}}(\varphi_{R}(X))\geq
E_{n}+\beta\mathrm{e}^{-\mu g_{1}(R)}+o(\mathrm{e}^{-\mu g_{1}(R)}
)+O(R^{-2/3})
\end{equation*}
as $R\rightarrow\infty$.
\end{proposition}
\begin{proof}
If $X=(x_{1},x_{2},\dots,x_{n})\in\partial\mathcal{U}_{1,R}$ there are
$i_{0},j_{0}\in\{1,2,\dots,n\}$ such that $\lvert x_{i_{0}}-x_{j_{0}}
\rvert=g_{1}(R).$ By Lemma \ref{lem:chains}, $d_{n}(i_{0},j_{0})=1$ for $R$
large enough. Then, assumption (H4) implies that
\begin{equation}
f(\overline{U}_{i_{0}})\overline{U}_{j_{0}}\leq0. \label{eq:117}
\end{equation}
On the other hand, it follows from Lemma \ref{decayU} and property
\eqref{eq:25} that there exist $r,\varepsilon>0$ such that $\left\vert
f(\overline{U}_{i_{0}})\right\vert \geq\varepsilon$ and $\left\vert
\overline{U}_{j_{0}}\right\vert \geq C_{1}\mathrm{e}^{-\mu g_{1}(R)}$ in
$B_{r}(x_{i_{0}})$ for $R$ large enough, independently of the choice of
$X\in\partial\mathcal{U}_{1,R}$. Hence
\begin{equation}
\frac{1}{2}\int_{\mathbb{R}^{N}}\left\vert f(\overline{U}_{i_{0}})\overline
{U}_{j_{0}}\right\vert \,\mathrm{d}x\geq\beta\mathrm{e}^{-\mu g_{1}(R)}
\label{eq:23}
\end{equation}
for some $\beta>0$ and large $R$. Moreover, Lemmas \ref{lem:chains},
\ref{lem:interaction-exponential} and \ref{decayU} yield
\begin{equation}
\int_{\mathbb{R}^{N}}\left\vert f(\overline{U}_{i})\overline{U}_{j}\right\vert
\,\mathrm{d}x=o(\mathrm{e}^{-\mu g_{1}(R)})\qquad\text{if }d_{n}(i,j)\geq2,
\label{eq:27}
\end{equation}
as $R\rightarrow\infty$.
Since $\overline{U}_{i}$ and $\overline{V}_{i}$ are uniformly bounded, using
Lemma~\ref{lem:splitting-f}, estimate \eqref{eq:13}, and Lemmas
\ref{lem:chains} and \ref{lem:interaction-exponential}, we obtain
\begin{align}
&  \left\vert \int_{\Omega_{R}}\left[  F(
{\sum_{i}}
\overline{V}_{i})-
{\sum_{i}}
F(\overline{V}_{i})\right]  -
{\sum_{i\neq j}}
\int_{\Omega_{R}}f(\overline{V}_{i})\overline{V}_{j}\right\vert \label{eq:107}
\\
&  \leq C
{\sum_{i<j}}
\int_{\Omega_{R}}\left\vert \overline{V}_{i}\overline{V}_{j}\right\vert
^{2\alpha}+C
{\sum_{i<j<k}}
\int_{\Omega_{R}}\left\vert \overline{V}_{i}\overline{V}_{j}\overline{V}
_{k}\right\vert ^{2/3}\nonumber\\
&  =C
{\sum_{i<j}}
\int_{\Omega_{R}}\left\vert \overline{U}_{i}\overline{U}_{j}\right\vert
^{2\alpha}+C
{\sum_{i<j<k}}
\int_{\Omega_{R}}\left\vert \overline{U}_{i}\overline{U}_{j}\overline{U}
_{k}\right\vert ^{2/3}+O(R^{-2/3})\nonumber\\
&  =o(\mathrm{e}^{-\mu g_{1}(R)})+O(R^{-2/3}).\nonumber
\end{align}
Therefore, using estimates \eqref{eq:17}, \eqref{eq:13}, \eqref{eq:117},
\eqref{eq:23} and \eqref{eq:27} we conclude that
\begin{align*}
J_{\Omega_{R}}(\varphi_{R}(X))=  &
{\sum_{i=1}^{n}}
J_{\Omega_{R}}(\overline{V}_{i})+\frac{1}{2}
{\sum_{i\neq j}}
\int_{\Omega_{R}}\left(  \nabla\overline{V}_{i}\cdot\nabla\overline{V}
_{j}+\lambda\overline{V}_{i}\overline{V}_{j}\right)  \,\mathrm{d}x\\
&  -\int_{\Omega_{R}}\left[  F(
{\sum_{i}}
\overline{V}_{i})-
{\sum_{i}}
F(\overline{V}_{i})\right]  \,\mathrm{d}x\\
=  &  E_{n}+\frac{1}{2}
{\sum_{i\neq j}}
\int_{\Omega_{R}}f(\overline{U}_{i})\overline{V}_{j}\,\mathrm{d}x-
{\sum_{i\neq j}}
\int_{\Omega_{R}}f(\overline{V}_{i})\overline{V}_{j}\,\mathrm{d}x\\
&  +o(\mathrm{e}^{-\mu g_{1}(R)})+O(R^{-2/3})\\
=  &  E_{n}-\frac{1}{2}
{\sum_{i\neq j}}
\int_{\Omega_{R}}f(\overline{U}_{i})\overline{U}_{j}\,\mathrm{d}
x+o(\mathrm{e}^{-\mu g_{1}(R)})+O(R^{-2/3})\\
\geq &  E_{n}+\frac{1}{2}\int_{\mathbb{R}^{N}}\left\vert f(\overline{U}
_{i_{0}})\overline{U}_{j_{0}}\right\vert \,\mathrm{d}x-\frac{1}{2}
{\sum_{d_{n}(i,j)\geq2}}
\int_{\Omega_{R}}\left\vert f(\overline{U}_{i})\overline{U}_{j}\right\vert
\,\mathrm{d}x\\
&  +o(\mathrm{e}^{-\mu g_{1}(R)})+O(R^{-2/3})\\
\geq &  E_{n}+\beta\mathrm{e}^{-\mu g_{1}(R)}+o(\mathrm{e}^{-\mu g_{1}
(R)})+O(R^{-2/3}).
\end{align*}
This asymptotic estimate is independent of $X$.
\end{proof}
\begin{proposition}
\label{prop:jr-low-in-interior} The estimate
\begin{equation*}
\inf_{X\in\mathcal{U}_{1,R}}J_{\Omega_{R}}(\varphi_{R}(X))\leq E_{n}
+o(\mathrm{e}^{-\mu g_{1}(R)})+O(R^{-2/3})
\end{equation*}
holds true as $R\rightarrow\infty$.
\end{proposition}
\begin{proof}
Fix $0<t_{1}<t_{2}<\dots<t_{n}<1$ and set $x_{R,i}:=R\gamma(t_{i})\in
\Gamma_{R}.$ By \eqref{eq:25}, $X_{R}:=(x_{R,1},x_{R,2},\dots,x_{R,n}
)\in\mathcal{U}_{1,R}$ for large $R$. As in the proof of
Proposition~\ref{prop:jr-high-on-boundary} we obtain
\begin{equation*}
J_{\Omega_{R}}(\varphi_{R}(X_{R}))=E_{n}-\frac{1}{2}
{\sum_{i\neq j}}
\int_{\mathbb{R}^{N}}f(\overline{U}_{i})\overline{U}_{j}\,\mathrm{d}
x+o(\mathrm{e}^{-\mu g_{1}(R)})+O(R^{-2/3}).
\end{equation*}
Choose $\varepsilon\in(0,\mu)$ and $\delta\in(0,{\min_{i\neq j}}\left\vert
\gamma(t_{i})-\gamma(t_{j})\right\vert )$. Then $\left\vert x_{R,i}
-x_{R,j}\right\vert \geq\delta R$ if $i\neq j$.
Lemma~\ref{lem:interaction-exponential} and Lemma \ref{decayU} imply that
\begin{equation*}
\int_{\mathbb{R}^{N}}f(\overline{U}_{i})\overline{U}_{j}\,\mathrm{d}
x=O(\mathrm{e}^{-(\mu-\varepsilon)\delta R})=o(\mathrm{e}^{-\mu g_{1}
(R)})\qquad\text{if }i\neq j,
\end{equation*}
and our claim follows.
\end{proof}
\subsection{The open-end tube case}
We now suppose that $\gamma(0)\neq\gamma(1).$ In this case we need also to
estimate the effect of the ends of the tubular domain on $V_{x,R}^{\pm}$. We
start by comparing the solutions $U^{\pm}$ to the limit problem in
$\mathbb{L}$\ with their projections onto a finite cylinder
\begin{equation*}
\mathbb{L}_{a,b}:=(-a,b)\times B_{1}^{N-1},\text{\qquad}a,b>0.
\end{equation*}
Let $\widetilde{U}_{a,b}^{\pm}$ be the unique solution of
\begin{equation}
\left\{
\begin{array}
[c]{ll}
-\Delta u+\lambda u=f(U^{\pm}) & \text{in }\mathbb{L}_{a,b}\text{,}\\
u=0 & \text{on }\partial\mathbb{L}_{a,b}\text{.}
\end{array}
\right.  \label{Utilde}
\end{equation}
Again, we consider $\widetilde{U}_{a,b}^{\pm}$ to be defined in $\mathbb{R}
^{N}$.
\begin{lemma}
\label{lem:tu-u-est-open}The inequalities
\begin{equation}
0\leq\widetilde{U}_{a,b}^{+}(\xi,\eta)\leq U^{+}(\xi,\eta),\qquad U^{-}
(\xi,\eta)\leq\widetilde{U}_{a,b}^{-}(\xi,\eta)\leq0, \label{eq:108}
\end{equation}
and
\begin{equation}
\lvert U^{\pm}(\xi,\eta)-\widetilde{U}_{a,b}^{\pm}(\xi,\eta)\rvert\leq
C_{2}\vartheta_{1,1}(\eta)\left(  \mathrm{e}^{-\mu(a+\lvert\xi+a\rvert
)}+\mathrm{e}^{-\mu(b+\lvert\xi-b\rvert)}\right)  \label{eq:43}
\end{equation}
hold true for all $(\xi,\eta)\in\mathbb{L}$, where $C_{2}$ is the same
constant as in \emph{Lemma \ref{decayU}}. Moreover, there are $C_{5},C_{6}>0$
such that
\begin{equation}
C_{5}\mathrm{e}^{-\mu\min\{a,b\}}\leq\lVert U^{\pm}-\widetilde{U}_{a,b}^{\pm
}\rVert_{H_{0}^{1}(\mathbb{L})}\leq C_{6}\mathrm{e}^{-\mu\min\{a,b\}}.
\label{eq:44}
\end{equation}
\end{lemma}
\begin{proof}
Note that $U^{\pm}$ and $\widetilde{U}_{a,b}^{\pm}$ are in $C^{2}
(\mathbb{L}_{a,b})\cap C(\overline{\mathbb{L}_{a,b}})$. Set $Y_{a,b}:=U^{\pm
}-\widetilde{U}_{a,b}^{\pm}.$ We claim that the inequalities
\begin{align}
&  C_{1}\vartheta_{1,1}(\eta)\max\{\mathrm{e}^{-\mu(a+\lvert\xi+a\rvert
)},\mathrm{e}^{-\mu(b+\lvert\xi-b\rvert)}\}\leq Y_{a,b}\label{Y}\\
&  \qquad\qquad\leq C_{2}\vartheta_{1,1}(\eta)(\mathrm{e}^{-\mu(a+\lvert
\xi+a\rvert)}+\mathrm{e}^{-\mu(b+\lvert\xi-b\rvert)})\nonumber
\end{align}
hold true for all $(\xi,\eta)\in\mathbb{L}$, where $C_{1}$ and $C_{2}$ are the
constants in Lemma \ref{decayU}. This is trivially true in $\mathbb{L}
\smallsetminus\mathbb{L}_{a,b}.$ For $(\xi,\eta)\in\mathbb{L}_{a,b}$ it
follows from the maximum principle, because the equalities
\begin{align*}
(-\Delta+\lambda)\vartheta_{1,1}(\eta)\mathrm{e}^{-\mu(a+\lvert\xi+a\rvert)}
&  =0,\\
(-\Delta+\lambda)\vartheta_{1,1}(\eta)\mathrm{e}^{-\mu(b+\lvert\xi-b\rvert)}
&  =0,\\
(-\Delta+\lambda)Y_{a,b}  &  =0,
\end{align*}
hold true in $\mathbb{L}_{a,b}$. Inequalities \eqref{eq:108} and \eqref{eq:43}
are now a consequence of \eqref{Y} and the maximum principle.
Next we prove \eqref{eq:44}. A straightforward computation using \eqref{Y}
yields
\begin{equation}
\left\Vert Y_{a,b}\right\Vert _{L^{2}(\mathbb{L})}=O(\mathrm{e}^{-\mu
\min\{a,b\}}) \label{eq:40}
\end{equation}
as $a,b\rightarrow\infty$. A standard regularity argument,
cf. \cite[Theorem~9.12]{MR737190}, yields
\begin{equation*}
\left\Vert \nabla Y_{a,b}\right\Vert _{L^{2}(\mathbb{L})}=O(\mathrm{e}
^{-\mu\min\{a,b\}}).
\end{equation*}
which, together with \eqref{eq:40}, this gives the inequality in the
right-hand side of \eqref{eq:44}. To prove the other inequality it is enough
to show that
\begin{equation}
\left\Vert Y_{a,b}\right\Vert _{L^{2}(\mathbb{L})}\geq C\mathrm{e}^{-\mu
\min\{a,b\}}, \label{eq:41}
\end{equation}
where $C$ is independent of $a$ and $b$. Note that
\begin{equation*}
a+\left\vert \xi+a\right\vert \leq b+\left\vert \xi-b\right\vert \quad\text{if
and only if}\quad\xi\leq b-a.
\end{equation*}
It follows that
\begin{align}
&  \int_{\mathbb{R}}\max\{\mathrm{e}^{-\mu(a+\lvert\xi+a\rvert)}
,\mathrm{e}^{-\mu(b+\lvert\xi-b\rvert)}\}\,\mathrm{d}\xi\label{eq:109}\\
&  =\frac{1}{\mu}(\mathrm{e}^{-2\mu a}-\mathrm{e}^{-2\mu(a+b)}+\mathrm{e}
^{-2\mu b})\nonumber\\
&  \geq\frac{1}{\mu}\max\{\mathrm{e}{^{-2\mu a},\mathrm{e}^{-2\mu b}\}}
=\frac{1}{\mu}\mathrm{e}^{-2\mu\min\{a,b\}},\nonumber
\end{align}
which together with \eqref{Y} yields \eqref{eq:41}. The proof is complete.
\end{proof}
Next we compare $V_{x,R}^{\pm}$ with the function
\begin{equation*}
W_{x,R}^{\pm}(y):=\widetilde{U}_{\lvert x-R\gamma(0)\rvert,\lvert
x-R\gamma(1)\rvert}^{\pm}(A_{x}(y-x)),\qquad y\in\mathbb{R}^{N},
\end{equation*}
with $A_{x}$ as in \eqref{eq:3} and $\widetilde{U}_{a,b}^{\pm}$ as in
\eqref{Utilde}. Thus, the support of $W_{x,R}^{\pm}$ is contained in a copy of
the finite cylinder $\mathbb{L}_{\lvert x-R\gamma(0)\rvert,\lvert
x-R\gamma(1)\rvert},$ obtained by translating $0$ to $x$ and identifying
$\mathbb{R}\times\{0\}$ with the tangent space to $\Gamma_{R}$ at $x.$
\begin{lemma}
\label{lem:w-v-est-open}For $s\in\lbrack1,r_{0})$ and $p\in(0,\infty)$ the
asymptotic estimates
\begin{align}
\int_{\mathbb{R}^{N}}\lvert V_{x,s,R}^{\pm}-W_{x,R}^{\pm}\rvert^{p}
\,\mathrm{d}y  &  =O(R^{-\min\{p,1\}}),\label{eq:36}\\
\int_{\mathbb{R}^{N}}\lvert\nabla V_{x,s,R}^{\pm}-\nabla W_{x,R}^{\pm}
\rvert^{2}\,\mathrm{d}y  &  =O(R^{-1}),\label{eq:37}\\
\int_{\mathbb{R}^{N}}\lvert F(V_{x,s,R}^{\pm})-F(W_{x,R}^{\pm})\rvert
\,\mathrm{d}y  &  =O(R^{-1}),\label{eq:38}\\
\int_{\mathbb{R}^{N}}\lvert f(V_{x,s,R}^{\pm})-f(W_{x,R}^{\pm})\rvert
^{p}\,\mathrm{d}y  &  =O(R^{-\min\{p,1\}}), \label{eq:39}
\end{align}
hold true as $R\rightarrow\infty,$ independently of $x\in\Gamma_{R}.$
\end{lemma}
\begin{proof}
Let $x_{R}\in\Gamma_{R}.$ If $x_{R}$ is far from the boundary the proof is
similar to that of Lemma~\ref{lem:u-v-est}, but if $x_{R}$ is close to the
boundary the proof requires some new geometric considerations. More precisely,
we consider two cases:
a) $\left\vert x_{R}-R\gamma(0)\right\vert \geq2R^{1/4}$ and $\left\vert
x_{R}-R\gamma(1)\right\vert \geq2R^{1/4}.$ Then the proof is the same as that
of Lemma~\ref{lem:u-v-est}.
b) Either $\left\vert x_{R}-R\gamma(0)\right\vert <2R^{1/4}$ or $\left\vert
x_{R}-R\gamma(1)\right\vert <2R^{1/4}$. Since both cases are similar, we only
consider the case
\begin{equation}
b_{R}:=\left\vert x_{R}-R\gamma(1)\right\vert <2R^{1/4}. \label{eq:65}
\end{equation}
For each $R$ we fix a coordinate system by identifying $x_{R}$ with $0$ and
the tangent space to $\Gamma_{R}$ at $x_{R}$ with $\mathbb{R}\times\{0\},$
preserving the orientation. In this coordinate system we consider the infinite
cylinder $\mathbb{L}$ and the finite cylinders
\begin{equation*}
\begin{array}
[c]{l}
\mathbb{L}_{R}:=\mathbb{L}_{\lvert x_{R}-R\gamma(0)\rvert,\lvert x_{R}
-R\gamma(1)\rvert},\medskip\\
Q_{R}:=(-R^{1/3},R^{1/3})\times B_{s}^{N-1}(0),
\end{array}
\end{equation*}
and we write $R\gamma(1)=(\xi_{R},\eta_{R}).$ Note that, since $\frac{x_{R}
}{R}\rightarrow\gamma(1)$ as $R\rightarrow\infty,$ the end of $\Omega_{R}$
which contains $R\gamma(0)$ lies outside of $Q_{R}$ for $R$ large enough.
We may assume that $\gamma$ is defined in some interval $(0,1+\varepsilon), $
$\varepsilon>0,$ and write $\widetilde{\Gamma}_{R}:=\{R\gamma(t)\mid
t\in\lbrack0,1+\varepsilon)\}$ and $\widetilde{\Omega}_{R}$ for its tubular
neighborhood of radius $1.$ Then $\widetilde{\Gamma}_{R}\cap Q_{R}$ is
contained in the graph of a $C^{3}$-function $h_{R}:(-R^{1/3},R^{1/3}
)\rightarrow\mathbb{R}^{N-1}$ for large $R$. As before, inequalities
\eqref{eq:55}\ hold for $h_{R}.$ Since $0\leq\xi_{R}\leq b_{R}$ and $b_{R}
^{2}-\xi_{R}^{2}=\eta_{R}^{2}=h_{R}(\xi_{R})^{2}$ we obtain
\begin{equation}
\left\vert b_{R}-\xi_{R}\right\vert =\frac{h_{R}(\xi_{R})^{2}}{\left\vert
b_{R}+\xi_{R}\right\vert }\leq\frac{C\xi_{R}}{2R^{2}}\leq C\frac{b_{R}}{R^{2}
}. \label{eq:73}
\end{equation}
Next, we express $\mathbb{R}^{N}$ as the union of the sets
\begin{equation}
\begin{array}
[c]{l}
D_{R}^{1}:=\mathbb{R}^{N}\smallsetminus Q_{R},\smallskip\\
D_{R}^{2}:=Q_{R}\cap\left[  \left(  \Omega_{R}\cup\mathbb{L}_{R}\right)
\smallsetminus(\widetilde{\Omega}_{R}\cap\mathbb{L})\right]  ,\smallskip\\
D_{R}^{3}:=Q_{R}\cap\left[  (\Omega_{R}\cap\left(  \mathbb{L}\smallsetminus
\mathbb{L}_{R}\right)  )\cup((\widetilde{\Omega}_{R}\smallsetminus\Omega
_{R})\cap\mathbb{L})\right]  ,\smallskip\\
D_{R}^{4}:=Q_{R}\cap\mathbb{L}_{R}\cap\Omega_{R},
\end{array}
\label{partition}
\end{equation}
and we show that the estimate \eqref{eq:36} holds true for the integral over
each one of these sets. Note that $D_{R}^{2}\subset Q_{R}\cap\lbrack
(\widetilde{\Omega}_{R}\cup\mathbb{L})\smallsetminus(\widetilde{\Omega}
_{R}\cap\mathbb{L})].$ Thus, the arguments for $D_{R}^{1}$ and $D_{R}^{2}$ are
the same as those\ given to prove Claims 1-3 in Lemma \ref{lem:u-v-est}. To
prove estimate \eqref{eq:36} over $D_{R}^{3},$ first observe that the angle
$\alpha_{R}$ between $\{b_{R}\}\times\mathbb{R}^{N-1}$ and the end of
$\Omega_{R}$ which contains $R\gamma(1)$ is the same as the angle between the
tangent space to $\Gamma_{R}$ at $x_{R},$ which we have identified with
$\mathbb{R}\times\{0\},$ and the tangent space to $\widetilde{\Gamma}_{R}$ at
$R\gamma(1).$ Therefore, using \eqref{eq:55} we obtain that
\begin{equation}
\tan\alpha_{R}=\left\vert h_{R}^{\prime}(\xi_{R})\right\vert \leq C\frac
{b_{R}}{R}. \label{eq:81}
\end{equation}
Since diam$(B_{1}^{N-1})=2$ it follows that
\begin{align}
D_{R}^{3}  &  \subset[\xi_{R}-2\tan\alpha_{R},\,b_{R}+2\tan\alpha_{R}]\times
B_{1}^{N-1}\label{eq:79}\\
&  \subset[b_{R}-s_{R},b_{R}+s_{R}]\times B_{1}^{N-1},\nonumber
\end{align}
where $s_{R}\geq0$ satisfies
\begin{equation}
s_{R}\leq C\left(  \frac{b_{R}}{R^{2}}+\frac{b_{R}}{R}\right)  \leq
C\frac{b_{R}}{R}. \label{eq:80}
\end{equation}
Here we have used \eqref{eq:73} and \eqref{eq:81}. Therefore, using Lemma
\ref{lem:decay-V}\ we conclude that
\begin{align}
\int_{D_{R}^{3}}\lvert V_{x,s,R}^{\pm}-W_{x,R}^{\pm}\rvert^{p}\,\mathrm{d}y
&  \leq C\int_{b_{R}-s_{R}}^{b_{R}+s_{R}}\mathrm{e}^{-pC_{4}\xi}
\,\mathrm{d}\xi\label{eq:83}\\
&  =C\mathrm{e}^{-pC_{4}b_{R}}\sinh(pC_{4}s_{R})\nonumber\\
&  \leq C\mathrm{e}^{-pC_{4}b_{R}}\frac{b_{R}}{R}=O(R^{-1}).\nonumber
\end{align}
for $R$ large enough. To prove estimate \eqref{eq:36} over $D_{R}^{4},$ we
start by estimating $\lvert V_{x,s,R}^{\pm}-W_{x,R}^{\pm}\rvert$ on $\partial
D_{R}^{4}.$ If $(\xi,\eta)\in\partial D_{R}^{4}\cap\partial\mathbb{L}$ then,
as in \eqref{eq:85}, we have that
\begin{equation*}
\text{dist}((\xi,\eta),\partial\widetilde{\Omega}_{R})\leq C\frac{1+\xi^{2}
}{R}.
\end{equation*}
Similarly, if $(\xi,\eta)\in\partial D_{R}^{4}\cap\partial\widetilde{\Omega
}_{R}$ then
\begin{equation*}
\text{dist}((\xi,\eta),\partial\mathbb{L})\leq C\frac{1+\xi^{2}}{R}.
\end{equation*}
Moreover, if $(\xi,\eta)\in\partial D_{R}^{4}\cap\partial\Omega_{R}
\cap\widetilde{\Omega}_{R}$ then
\begin{equation*}
\text{dist}((\xi,\eta),\partial\mathbb{L}_{R}\cap\mathbb{L})\leq2s_{R}\leq
C\frac{b_{R}}{R}\leq C\frac{\xi+s_{R}}{R}\leq C\left(  \frac{\xi}{R}
+\frac{b_{R}}{R^{2}}\right)  \leq C\frac{1+\xi^{2}}{R}.
\end{equation*}
Similarly, if $(\xi,\eta)\in\partial D_{R}^{4}\cap\partial\mathbb{L}_{R}
\cap\mathbb{L}$ then
\begin{equation*}
\text{dist}((\xi,\eta),\partial\Omega_{R}\cap\widetilde{\Omega}_{R})\leq
C\frac{1+\xi^{2}}{R}.
\end{equation*}
Since $V_{x,s,R}^{\pm}=0$ in $\mathbb{R}^{N}\smallsetminus\Omega_{R}$ and
$W_{x,R}^{\pm}=0$ in $\mathbb{R}^{N}\smallsetminus\mathbb{L}_{R},$ arguing as
in the proof of Claim 4 of Lemma \ref{lem:u-v-est}, we conclude that
\begin{equation}
\lvert V_{x,s,R}^{\pm}-W_{x,R}^{\pm}\rvert\leq C\mathrm{e}^{-C_{4}\left\vert
\xi\right\vert }\frac{1+\xi^{2}}{R}\qquad\text{on $\partial$}D_{R}^{4}\text{,}
\label{eq:84}
\end{equation}
and that
\begin{equation}
\int_{D_{R}^{4}}\lvert V_{x,s,R}^{\pm}-W_{x,R}^{\pm}\rvert^{p}\,\mathrm{d}
y=O(R^{-p}). \label{eq:82}
\end{equation}
This finishes the proof of \eqref{eq:36}.
The proof of \eqref{eq:37} is analogous to that of \eqref{eq:12}, using the
partition \eqref{partition}. Equations \eqref{eq:38} and \eqref{eq:39} follow
from \eqref{eq:36} as in the proof of Lemma~\ref{lem:u-v-est}.
\end{proof}
Again, we consider functions $g_{m}\colon\mathbb{R}^{+}\rightarrow
\mathbb{R}^{+}$ (to be fixed later) satisfying \eqref{eq:24}-\eqref{eq:25},
but this time we define $D_{m,R}$ as the set of points $(x_{1},x_{2}
,\dots,x_{n})$ in $(\Gamma_{R})^{n}$ such that either there exist
$i,j\in\{1,2,\dots,n\}$ with $i\neq j$ and $\lvert x_{i}-x_{j}\rvert\leq
g_{m}(R)$, or there exists $i\in\{1,2,\dots,n\}$ with $2$dist$(x_{i}
,\partial\Gamma_{R})\leq g_{m}(R).$ Then we define
\begin{equation}
\mathcal{U}_{m,R}:=\{(x_{1},x_{2},\dots,x_{n})\in(\Gamma_{R})^{n}
\smallsetminus D_{m,R}\mid(x_{1},x_{2},\dots,x_{n})\text{ is an $n $-chain}\}.
\label{Uopen}
\end{equation}
\begin{lemma}
\label{lem:nearness-w-v-open}The estimates
\begin{align}
\sup_{x\in\Gamma_{R}}\left\Vert V_{x,R}^{\pm}-W_{x,R}^{\pm}\right\Vert
_{H_{0}^{1}(\mathbb{R}^{N})}  &  =O(R^{-1/2}),\label{eq:28}\\
\sup_{x\in\Gamma_{R}}\left\vert J_{\Omega_{R}}(V_{x,R}^{\pm})-J_{\mathbb{L}
}(W_{x,R}^{\pm})\right\vert  &  =O(R^{-1}),\label{eq:45}\\
\sup_{\substack{x\in\Gamma_{R} \\\text{\emph{dist}}(x,\partial\Gamma_{R})\geq
g_{2}(R)/2}}\left\Vert \nabla J_{\Omega_{R}}(V_{x,R}^{\pm})\right\Vert
_{H_{0}^{1}(\Omega_{R})}  &  =O(R^{-1/2})+O(\mathrm{e}^{-\min\{p_{1},2\}\mu
g_{2}(R)/2}) \label{eq:46}
\end{align}
hold true as $R\rightarrow\infty$.
\end{lemma}
\begin{proof}
Estimates \eqref{eq:28} and \eqref{eq:45} follow immediately from
Lemma~\ref{lem:w-v-est-open}. To prove \eqref{eq:46} we first observe that
$|tU^{\pm}+(1-t)\widetilde{U}_{a,b}^{\pm}|\leq\left\vert U^{\pm}\right\vert $
for every $t\in\lbrack0,1]$ by \eqref{eq:108}. Moreover, (H3) implies that
$\left\vert f^{\prime}(u)\right\vert \leq C\left\vert u\right\vert ^{p_{1}-1}
$ for some constant $C$ which depends only on an upper bound for $\left\vert
u\right\vert $. Therefore Lemma~\ref{decayU} and inequality \eqref{eq:43}
imply
\begin{align*}
\int_{\mathbb{L}}|f(U^{\pm})-f(\widetilde{U}_{a,b}^{\pm})|^{2}\,\mathrm{d}x
&  \leq\int_{\mathbb{L}}\left(  \int_{0}^{1}|f^{\prime}(tU^{\pm}
)+(1-t)\widetilde{U}_{a,b}^{\pm}|\,\mathrm{d}t\right)  ^{2}|U^{\pm}
-\widetilde{U}_{a,b}^{\pm}|^{2}\,\mathrm{d}x\\
&  \leq C\int_{\mathbb{R}}\mathrm{e}^{-2(p_{1}-1)\mu}\mathrm{e}^{-2\mu
(a+\lvert\xi+a\rvert)}+\mathrm{e}^{-2\mu(b+\lvert\xi-b\rvert)}\,\mathrm{d}
\xi\\
&  =O(\mathrm{e}^{-2\min\{p_{1},2\}\mu\min\{a,b\}}),
\end{align*}
as $a,b\rightarrow\infty$. Therefore,
\begin{align*}
\left\Vert f(U_{x,R}^{\pm})-f(V_{x,R}^{\pm})\right\Vert _{L^{2}}  &
\leq\left\Vert f(U_{x,R}^{\pm})-f(W_{x,R}^{\pm})\right\Vert _{L^{2}
}+\left\Vert f(W_{x,R}^{\pm})-f(V_{x,R}^{\pm})\right\Vert _{L^{2}}\\
&  \leq O(\mathrm{e}^{-\min\{p_{1},2\}\mu g_{2}(R)/2})+O(R^{-1/2}),
\end{align*}
as $R\rightarrow\infty$. Arguing as in the proof of \eqref{eq:18},\ using this
estimate, we obtain \eqref{eq:46}.
\end{proof}
Define $\varphi_{R}\colon\mathcal{U}_{2,R}\rightarrow H_{0}^{1}(\Omega_{R})$
by
\begin{equation}
\varphi_{R}(X):=
{\sum_{i=1}^{k}}
(V_{x_{2i-1},R}^{+}+V_{x_{2i},R}^{-})+(n-2k)V_{x_{n},R}^{+},\qquad
X=(x_{1},x_{2},\dots,x_{n}), \label{phiopen}
\end{equation}
where $k$ is the largest integer smaller than or equal to $\frac{n}{2}.$ This
time we do not require that $n$ is even.
Next we show that the statements of Propositions
\ref{prop:lyapunov-gradient-estimate}--\ref{prop:jr-low-in-interior} are also
true for these new data. We set $\overline{U}_{i}$ and $\overline{V}_{i}$ as
in (\ref{bar}). Similarly, we set
\begin{equation*}
\overline{W}_{i}:=\left\{
\begin{array}
[c]{ll}
W_{x_{i},R}^{+} & \text{if }i\text{ is odd,}\\
W_{x_{i},R}^{-} & \text{if }i\text{ is even.}
\end{array}
\right.
\end{equation*}
\begin{proposition}
\label{prop:lyapunov-gradient-estimate-open}Let $\alpha$ be as in
\emph{Lemma~\ref{lem:splitting-f}} and fix $\alpha^{\prime}\in(1/2,\min
\{\alpha,p_{1}/2,1\})$. Then
\begin{equation*}
\sup_{X\in\mathcal{U}_{2,R}}\lVert\nabla J_{\Omega_{R}}(\varphi_{R}
(X))\rVert_{H_{0}^{1}(\Omega_{R})}=O(\mathrm{e}^{-\alpha^{\prime}\mu g_{2}
(R)})+O(R^{-1/2})
\end{equation*}
as $R\rightarrow\infty$.
\end{proposition}
\begin{proof}
The proof is completely analogous to that of
Proposition~\ref{prop:lyapunov-gradient-estimate}, using this time Lemmas
\ref{lem:nearness-w-v-open}\ and \ref{lem:w-v-est-open}.
\end{proof}
Set
\begin{equation*}
E_{n}:=k\left[  J_{\mathbb{L}}(U^{+})+J_{\mathbb{L}}(U^{-})\right]
+(n-2k)J_{\mathbb{L}}(U^{+}).
\end{equation*}
\begin{proposition}
\label{prop:jr-high-on-boundary-open}There exists $\beta>0$ such that
\begin{equation*}
\inf_{X\in\partial\mathcal{U}_{1,R}}J_{\Omega_{R}}(\varphi_{R}(X))\geq
E_{n}+\beta\mathrm{e}^{-\mu g_{1}(R)}+o(\mathrm{e}^{-\mu g_{1}(R)}
)+O(R^{-2/3})
\end{equation*}
as $R\rightarrow\infty$.
\end{proposition}
\begin{proof}
Let $X=(x_{1},x_{2},\dots,x_{n})\in\partial\mathcal{U}_{1,R}.$ The proof is
similar to that of Proposition \ref{prop:jr-high-on-boundary} except that now
we must replace $\overline{U}_{i}$ by $\overline{W}_{i}.$ So, in order to
arrive to the conclusion, we need the following estimates:
\begin{align}
J_{x_{i}+A_{x_{i}}^{-1}\mathbb{L}}(\overline{W}_{i})  &  \geq J_{x_{i}
+A_{x_{i}}^{-1}\mathbb{L}}(\overline{U}_{i})+C\mathrm{e}^{-2\mu\text{dist}
(x_{i},\partial\Gamma_{R})},\label{ener}\\
\int_{\mathbb{R}^{N}}f(\overline{W}_{i})\overline{W}_{j}  &  =\int
_{\mathbb{R}^{N}}f(\overline{U}_{i})\overline{U}_{j}+o(\mathrm{e}^{-\mu
g_{1}(R)}). \label{inter}
\end{align}
Let us prove the first one. After an appropriate change of coordinates
$\overline{U}_{i}$ becomes $U^{\pm}$\ and $\overline{W}_{i}$ becomes
$\widetilde{U}_{a,b}^{\pm}.$ Recall that $J_{\mathbb{L}}^{\prime}(U^{\pm}
)=0$\ and observe that $\left\vert f^{\prime}(U^{\pm}(\xi,\eta))\right\vert
\leq C\mathrm{e}^{-(p_{1}-1)\mu\left\vert \xi\right\vert }$ due to condition
(H3) and Lemma~\ref{decayU}. So using Lemma \ref{lem:tu-u-est-open}\ we obtain
\begin{align*}
J_{\mathbb{L}}(\widetilde{U}_{a,b}^{\pm})  &  =J_{\mathbb{L}}(U^{\pm}
)+\frac{1}{2}J_{\mathbb{L}}^{\prime\prime}(U^{\pm})[\widetilde{U}_{a,b}^{\pm
}-U^{\pm},\widetilde{U}_{a,b}^{\pm}-U^{\pm}]+o(\lVert U^{\pm}-\widetilde
{U}_{a,b}^{\pm}\rVert_{H_{0}^{1}(\mathbb{L})}^{2})\\
&  \geq J_{\mathbb{L}}(U^{\pm})+\frac{1}{2}\lVert U^{\pm}-\widetilde{U}
_{a,b}^{\pm}\rVert_{H_{0}^{1}(\mathbb{L})}^{2}+o(\lVert U^{\pm}-\widetilde
{U}_{a,b}^{\pm}\rVert_{H_{0}^{1}(\mathbb{L})}^{2})\\
&  \geq J_{\mathbb{L}}(U^{\pm})+C\mathrm{e}^{-2\mu\min\{a,b\}}
\end{align*}
for $R$ large enough. This proves (\ref{ener}).
To prove the second estimate it suffices to show that
\begin{align}
\int_{\mathbb{R}^{N}}({f(\overline{W}_{i})-f(\overline{U}_{i})}\overline
{W}_{j}  &  =o(\mathrm{e}^{-\mu g_{1}(R)})\label{eq:112}\\
\int_{\mathbb{R}^{N}}f(\overline{U}_{i})(\overline{W}_{j}-\overline{U}_{j})
&  =o(\mathrm{e}^{-\mu g_{1}(R)}) \label{eq:113}
\end{align}
as $R\rightarrow\infty$. Since the proof of both estimates is similar, we only
prove \eqref{eq:112}. After a change of coordinates we may assume that
$x_{i}=0$ and that the tangent space to $\Gamma_{R}$ at $x_{i}$ is
$\mathbb{R}\times\{0\}$. Then we set $a:=\left\vert R\gamma(0)\right\vert $
and $b:=\left\vert R\gamma(1)\right\vert .$ We may assume without loss of
generality that $a\leq b.$ Since $\left\vert \overline{W}_{j}(x)\right\vert
\leq C\mathrm{e}^{-\mu\left\vert x-x_{j}\right\vert }$ by \eqref{eq:108} and
Lemma~\ref{decayU}, the proof of \eqref{eq:112}\ reduces to showing that
\begin{equation}
\int_{\mathbb{L}}\left\vert f(U^{\pm}(x))-f(\widetilde{U}_{a,b}^{\pm
}(x))\right\vert \mathrm{e}^{-\mu\left\vert x-x_{j}\right\vert }
\mathrm{d}x=o(\mathrm{e}^{-\mu g_{1}(R)}) \label{eq:115}
\end{equation}
as $R\rightarrow\infty$. We distinguish two cases: If $\left\vert
x_{j}\right\vert \geq2g_{1}(R),$ using condition (H3),
Lemma~\ref{lem:interaction-exponential} and \eqref{eq:10} we obtain
\begin{align*}
\int_{\mathbb{L}}\left\vert f(U^{\pm}(x))-f(\widetilde{U}_{a,b}^{\pm
}(x))\right\vert \mathrm{e}^{-\mu\left\vert x-x_{j}\right\vert }\mathrm{d}x
&  \leq C\int_{\mathbb{L}}\mathrm{e}^{-p_{1}\mu\left\vert x\right\vert
}\mathrm{e}^{-\mu\left\vert x-x_{j}\right\vert }\mathrm{d}x\\
&  \leq C\mathrm{e}^{-\mu\left\vert x_{j}\right\vert }\leq C\mathrm{e}^{-2\mu
g_{1}(R)}=o(\mathrm{e}^{-\mu g_{1}(R)})
\end{align*}
as $R\rightarrow\infty$. On the other hand, if $\left\vert x_{j}\right\vert
\leq2g_{1}(R)$ we write $x_{j}=(\xi_{j},\eta_{j})$ and use the Lipschitz
continuity of $f$ on bounded sets and \eqref{eq:43} to obtain
\begin{align*}
&  \int_{\mathbb{L}}\left\vert f(U^{\pm}(x))-f(\widetilde{U}_{a,b}^{\pm
}(x))\right\vert \mathrm{e}^{-\mu\left\vert x-x_{j}\right\vert }\mathrm{d}x\\
&  \leq C\int_{\mathbb{R}}\left(  \mathrm{e}^{-\mu(a+\lvert\xi+a\rvert
)}+\mathrm{e}^{-\mu(b+\lvert\xi-b\rvert)}\right)  \mathrm{e}^{-\mu\left\vert
\xi-\xi_{j}\right\vert }\mathrm{d}\xi\\
&  \leq C\left(  \mathrm{e}^{-\mu(a+\lvert\xi_{j}+a\rvert)}+\mathrm{e}
^{-\mu(\lvert\xi_{j}-b\rvert)}\right) \\
&  =C\mathrm{e}^{-\mu(a+\lvert\xi_{j}+a\rvert)}+o(\mathrm{e}^{-\mu g_{1}(R)})
\end{align*}
The last equality follows from $\left\vert x_{j}\right\vert \leq2g_{1}(R),$
$b:=\left\vert R\gamma(1)\right\vert $ and \eqref{eq:25}. Now, if $j>i$ we
have that $a\geq\frac{3}{2}g_{1}(R)(1+o(1))$, and if $j<i$ we have that
$\xi_{j}+a\geq\frac{3}{2}g_{1}(R)(1+o(1))$ as $R\rightarrow\infty$. So in both
cases $\mathrm{e}^{-\mu(a+\lvert\xi_{j}+a\rvert)}=o(\mathrm{e}^{-\mu g_{1}
(R)})$. This proves \eqref{eq:115} and, hence, \eqref{eq:112}.
Now we may argue as in Proposition \ref{prop:jr-high-on-boundary}. The
analogue of \eqref{eq:107} with ${\overline{U}_{i}}$ replaced by
${\overline{W}_{i}}$ is obtained in a similar way. Therefore, using estimates
\eqref{ener}, \eqref{inter}\ and \eqref{d>2} we conclude that
\begin{align*}
J_{\Omega_{R}}(\varphi_{R}(X))=  &
{\sum_{i=1}^{n}}
J_{\Omega_{R}}(\overline{V}_{i})+\frac{1}{2}
{\sum_{i\neq j}}
\int_{\Omega_{R}}f(\overline{U}_{i})\overline{V}_{j}\,\mathrm{d}x-
{\sum_{i\neq j}}
\int_{\Omega_{R}}f(\overline{V}_{i})\overline{V}_{j}\,\mathrm{d}x\\
&  +o(\mathrm{e}^{-\mu g_{1}(R)})+O(R^{-2/3})\\
=  &
{\sum_{i=1}^{n}}
J_{x_{i}+A_{x_{i}}^{-1}\mathbb{L}}(\overline{W}_{i})+\frac{1}{2}
{\sum_{i\neq j}}
\int_{\Omega_{R}}f(\overline{U}_{i})\overline{W}_{j}\,\mathrm{d}x-
{\sum_{i\neq j}}
\int_{\Omega_{R}}f(\overline{W}_{i})\overline{W}_{j}\,\mathrm{d}x\\
&  +o(\mathrm{e}^{-\mu g_{1}(R)})+O(R^{-2/3})\\
\geq &  E_{n}+C
{\sum_{i=1}^{n}}
\mathrm{e}^{-2\mu\text{dist}(x_{i},\partial\Gamma_{R})}+\frac{1}{2}
{\sum_{\left\vert i-j\right\vert =1}}
\int_{\Omega_{R}}\left\vert f(\overline{U}_{i})\overline{U}_{j}\right\vert
\,\mathrm{d}x\\
&  +o(\mathrm{e}^{-\mu g_{1}(R)})+O(R^{-2/3}).
\end{align*}
Since $X\in\partial\mathcal{U}_{1,R},$ either dist$(x_{1},\partial\Gamma
_{R})=g_{1}(R)/2$ or dist$(x_{n},\partial\Gamma_{R})=g_{1}(R)/2$ or
$\left\vert x_{i+1}-x_{i}\right\vert =g_{1}(R)$ for some $i=1,...,n-1.$ In any
case, our claim follows.
\end{proof}
\begin{proposition}
\label{prop:jr-low-in-interior-open}The estimate
\begin{equation*}
\inf_{X\in\mathcal{U}_{1,R}}J_{\Omega_{R}}(\varphi_{R}(X))\leq E_{n}
+o(\mathrm{e}^{-\mu g_{1}(R)})+O(R^{-2/3})
\end{equation*}
holds true as $R\rightarrow\infty$.
\end{proposition}
\begin{proof}
The proof is similar to that of Proposition \ref{prop:jr-low-in-interior},
this time taking into account that dist$(x_{i},\partial\Gamma_{R})\geq CR$ for
some $C>0$ and every $R$ and $i.$
\end{proof}
\section{Proof of the Main Results}
\label{sec:finite-dimens-reduct}
\subsection{The Finite Dimensional Reduction}
\label{sec:finite-dimens-reduct-1}
Let $\mathcal{U}_{2,R}$ and $\varphi_{R}:\mathcal{U}_{2,R}\rightarrow
H_{0}^{1}(\Omega_{R})$ be as in (\ref{Uclosed}) and (\ref{immersion}) when
$\Gamma$ is a closed curve and as in (\ref{Uopen}) and (\ref{phiopen}) if
$\gamma(0)\neq\gamma(1).$ Set
\begin{equation*}
\Sigma_{R}:=\varphi_{R}(\mathcal{U}_{2,R}).
\end{equation*}
\begin{lemma}
\label{lem:phi-immersion} $\Sigma_{R}$ is a finite dimensional $C^{2}
$-submanifold of $H_{0}^{1}(\Omega_{R})$.
\end{lemma}
\begin{proof}
It is easy to see that the map $\varphi_{R}$ is a $C^{2}$-immersion. If
$\partial\Gamma\neq\emptyset$ or $n\leq2$ then $\varphi_{R}$ is injective, and
hence $\Sigma_{R}$ is a submanifold of $H_{0}^{1}(\Omega_{R})$. On the other
hand, if $\partial\Gamma=\emptyset$ and $n\geq4$ then $\varphi_{R}$ is not
injective: two points in $\mathcal{U}_{2,R}$ have the same image under
$\varphi_{R}$ if and only if one of them is obtained from the other after a
finite number of shifts of the form $x_{i}\mapsto x_{i+2}$. Since the group of
permutations acts freely on $\mathcal{U}_{2,R},$ $\Sigma_{R}$ is a submanifold
of $H_{0}^{1}(\Omega_{R})$ also in this case.
\end{proof}
We shall reduce the problem of finding a critical point of $J_{\Omega_{R}}$ to
that of finding a critical point of a function $G_{R}:\Sigma_{R}
\rightarrow\mathbb{R}$, which will be defined below.
For $u\in\Sigma_{R}$ we denote by $T_{u}\Sigma_{R}$ the tangent space to
$\Sigma_{R}$ at $u$, by $T_{u}^{\perp}\Sigma_{R}$ its orthogonal complement in
$H_{0}^{1}(\Omega_{R})$ and by $P_{u,R}^{\perp}:H_{0}^{1}(\Omega
_{R})\rightarrow T_{u}^{\perp}\Sigma_{R}$ the orthogonal projection. We
consider $\mathrm{D}^{2}J_{\Omega_{R}}(u)$ as the derivative of the gradient
vector field $\nabla J_{\Omega_{R}}:H_{0}^{1}(\Omega_{R})\rightarrow H_{0}
^{1}(\Omega_{R})$ at $u,$ and define
\begin{equation*}
L_{u,R}:=P_{u,R}^{\perp}\mathrm{D}^{2}J_{\Omega_{R}}(u)|_{T_{u}^{\perp}
\Sigma_{R}}:T_{u}^{\perp}\Sigma_{R}\rightarrow T_{u}^{\perp}\Sigma_{R}.
\end{equation*}
We write $\mathcal{L}(T_{u}^{\perp}\Sigma_{R})$ for the space of bounded
linear operators from $T_{u}^{\perp}\Sigma_{R}$ into itself.
\begin{lemma}
\label{prop:lyapunov-inverse-estimate} If $R$ is large enough and $u\in
\Sigma_{R}$, then $L_{u,R}$ is invertible in $\mathcal{L}(T_{u}^{\perp}
\Sigma_{R})$ and
\begin{equation*}
\limsup_{R\rightarrow\infty}\sup_{u\in\Sigma_{R}}\left\Vert L_{u,R}
^{-1}\right\Vert _{\mathcal{L}(T_{u}^{\perp}\Sigma_{R})}<\infty.
\end{equation*}
\end{lemma}
\begin{proof}
  The proof of this fact is standard, see for example Lemma~3.8(v) in \cite{MR2216902}.
\end{proof}
\begin{lemma}
\label{lem:existence-reduction} There exist $r_{0}>0$ and $R_{1}\geq1$ such
that for $R\geq R_{1}$ and for every $u\in\Sigma_{R}$ there is a unique
$v_{u}\in u+T_{u}^{\perp}\Sigma_{R}$ which satisfies $\lVert u-v_{u}
\rVert_{H_{0}^{1}(\Omega_{R})}<r_{0}$ and $P_{u,R}^{\perp}\nabla J_{\Omega
_{R}}(v_{u})=0$. The estimates
\begin{equation}
\lVert u-v_{u}\rVert_{H_{0}^{1}(\Omega_{R})}=O(\lVert\nabla J_{\Omega_{R}
}(u)\rVert_{H_{0}^{1}(\Omega_{R})}) \label{eq:30}
\end{equation}
and
\begin{equation}
\lvert J_{\Omega_{R}}(u)-J_{\Omega_{R}}(v_{u})\rvert=O(\lVert\nabla
J_{\Omega_{R}}(u)\rVert_{H_{0}^{1}(\Omega_{R})}^{2}) \label{eq:31}
\end{equation}
hold true as $R\rightarrow\infty$, independently of $u\in\Sigma_{R}$.
Moreover, the operator
\begin{equation*}
P_{u,R}^{\perp}\mathrm{D}^{2}J_{\Omega_{R}}(v_{u})|_{T_{u}^{\perp}\Sigma_{R}
}:T_{u}^{\perp}\Sigma_{R}\rightarrow T_{u}^{\perp}\Sigma_{R}
\end{equation*}
is invertible in $\mathcal{L}(T_{u}^{\perp}\Sigma_{R})$.
\end{lemma}
\begin{proof}
Along this proof $B_{r}Z$ will denote the open ball of radius $r$ centered at
$0$ in a normed space $Z$, and $\overline{B}_{r}Z$ will denote its closure.
By Lemma~\ref{prop:lyapunov-inverse-estimate} we may fix $M\geq1$ satisfying
\begin{equation}
M>\limsup_{R\rightarrow\infty}\sup_{u\in\Sigma_{R}}\left\Vert L_{u,R}
^{-1}\right\Vert _{\mathcal{L}(T_{u}^{\perp}\Sigma_{R})}. \label{eq:119}
\end{equation}
Clearly,
\begin{equation*}
C_{0}:=\limsup_{R\rightarrow\infty}\sup_{u\in\Sigma_{R}}\left\Vert
u\right\Vert _{H_{0}^{1}(\Omega_{R})}<\infty.
\end{equation*}
Condition~(H3) yields
\begin{equation}
\limsup_{R\rightarrow\infty}\left\Vert J_{\Omega_{R}}\right\Vert
_{C^{2,\bar{\alpha}}(B_{2C_{0}}H_{0}^{1}(\Omega_{R}))}<\infty\label{eq:8}
\end{equation}
for some $\bar{\alpha}\in(0,1]$.
By Lemma~\ref{prop:lyapunov-inverse-estimate} and \eqref{eq:8} there is
$r_{0}>0$ such that for $R$ large enough
\begin{equation}
\left\Vert \mathrm{D}^{2}J_{\Omega_{R}}(u)-\mathrm{D}^{2}J_{\Omega_{R}
}(v)\right\Vert _{\mathcal{L}(H_{0}^{1}(\Omega_{R}))}\leq\frac{1}{2M}
\label{eq:32}
\end{equation}
and $P_{u,R}^{\perp}\mathrm{D}^{2}J_{\Omega_{R}}(v)|_{T_{u}^{\perp}\Sigma_{R}
}$ is invertible in $\mathcal{L}(T_{u}^{\perp}\Sigma_{R}),$ for every
$u\in\Sigma_{R}$ and $v\in H_{0}^{1}(\Omega_{R})$ with $\left\Vert
u-v\right\Vert _{H_{0}^{1}(\Omega_{R})}\leq r_{0}$. Moreover, for $R$ large
enough,
\begin{equation}
\sup_{u\in\Sigma_{R}}\lVert\nabla J_{\Omega_{R}}(u)\rVert_{H_{0}^{1}
(\Omega_{R})}\leq\frac{r_{0}}{2M}. \label{eq:7}
\end{equation}
because of Propositions \ref{prop:lyapunov-gradient-estimate} and
\ref{prop:lyapunov-gradient-estimate-open}. Fix $u\in\Sigma_{R}$ and define
$g\colon T_{u}^{\perp}\Sigma_{R}\rightarrow T_{u}^{\perp}\Sigma_{R}$ by
\begin{equation*}
g(w):=w-L_{u,R}^{-1}P_{u,R}^{\perp}\nabla J_{\Omega_{R}}(u+w).
\end{equation*}
If $w\in\overline{B}_{r_{0}}T_{u}^{\perp}\Sigma_{R}$ it follows from
\eqref{eq:32} and \eqref{eq:7} that
\begin{align}
\left\Vert g(w)\right\Vert  &  \leq M\left\Vert \mathrm{D}^{2}J_{\Omega_{R}
}(u)w-\nabla J_{\Omega_{R}}(u+w)\right\Vert \nonumber\\
&  =M\left\Vert -\nabla J_{\Omega_{R}}(u)-\int_{0}^{1}(\mathrm{D}^{2}
J_{\Omega_{R}}(u+tw)-\mathrm{D}^{2}J_{\Omega_{R}}(u))w\text{ \textrm{d}
}t\right\Vert \label{eq:6}\\
&  \leq M\left(  \left\Vert \nabla J_{\Omega_{R}}(u)\right\Vert +\frac
{\left\Vert w\right\Vert }{2M}\right)  \leq r_{0}.\nonumber
\end{align}
Hence, $g$ maps $\overline{B}_{r_{0}}T_{u}^{\perp}\Sigma_{R}$ into itself.
Moreover, by \eqref{eq:119} and \eqref{eq:32} we have
\begin{equation*}
\left\Vert g^{\prime}(w)\right\Vert \leq\left\Vert L_{u,R}^{-1}\right\Vert
\left\Vert \mathrm{D}^{2}J_{\Omega_{R}}(u)-\mathrm{D}^{2}J_{\Omega_{R}
}(u+w)\right\Vert \leq\frac{1}{2}.
\end{equation*}
Therefore, $g$ is a contraction on $\overline{B}_{r_{0}}T_{u}^{\perp}
\Sigma_{R}$ and by Banach's fixed point theorem $g$ has a unique fixed point
$w_{u}\in\overline{B}_{r_{0}}T_{u}^{\perp}\Sigma_{R}$. Thus, $v_{u}:=u+w_{u}$
is then the only zero of $P_{u,R}^{\perp}\nabla J_{\Omega_{R}}$ in
$u+\overline{B}_{r_{0}}T_{u}^{\perp}\Sigma_{R}.$
Inequality \eqref{eq:6} with $w:=w_{u}=g(w_{u})$ yields $\left\Vert
w_{u}\right\Vert \leq2M\left\Vert \nabla J_{\Omega_{R}}(u)\right\Vert $ and
hence \eqref{eq:30}. Moreover, since $\mathrm{D}J_{\Omega_{R}}(v_{u}
)[w_{u}]=0,$
\begin{align}
&  \lvert J_{\Omega_{R}}(u)-J_{\Omega_{R}}(v_{u})\rvert\label{eq:34}\\
&  \leq\left\vert \mathrm{D}J_{\Omega_{R}}(v_{u})[w_{u}]\right\vert +\int
_{0}^{t}(1-t)\left\vert \mathrm{D}^{2}J_{\Omega_{R}}(u+(1-t)w_{u})\left[
w_{u},w_{u}\right]  \right\vert \text{ \textrm{d}}t\nonumber\\
&  \leq C\left\Vert w_{u}\right\Vert ^{2}\nonumber
\end{align}
for some constant $C$ independent of $R$ and $u$. Now \eqref{eq:30} and
\eqref{eq:34} imply \eqref{eq:31}. Finally, if $R$ is large enough,
\eqref{eq:30} implies the strict inequality $\lVert u-v_{u}\rVert_{H_{0}
^{1}(\Omega_{R})}<r_{0}$, as stated in the lemma.
\end{proof}
We now fix $r_{0}$ and $R_{1}$ as in Lemma~\ref{lem:existence-reduction}. If
$R\geq R_{1}$ we define $G_{R}\colon\Sigma_{R}\rightarrow\mathbb{R}$ by
\begin{equation*}
G_{R}(u):=J_{\Omega_{R}}(v_{u}).
\end{equation*}
where $v_{u}$ is given by Lemma~\ref{lem:existence-reduction}.
\begin{proposition}
\label{lem:reduced-problem} For $R\geq R_{1}$ the map $G_{R}$ is in
$C^{1}(\Sigma_{R},\mathbb{R})$. If $u\in\Sigma_{R}$ is a critical point of
$G_{R}$ then $v_{u}$ is a critical point of $J_{\Omega_{R}}$.
\end{proposition}
\begin{proof}
The map $u\mapsto v_{u}$ is a cross section of the normal disc bundle of
radius $r_{0}$ over $\Sigma_{R},$ so its image $\tilde{\Sigma}_{R}
:=\{v_{u}:u\in\Sigma_{R}\}$ is a submanifold which is transversal to the
fibres, that is, $H_{0}^{1}(\Omega)=T_{v_{u}}\tilde{\Sigma}_{R}\oplus
T_{u}^{\perp}\Sigma_{R}.$ The map $\psi_{R}:\Sigma_{R}\rightarrow\tilde
{\Sigma}_{R}$ given by $\psi_{R}(u):=v_{u}$ is a $C^{1}$-diffeomorphism.
Therefore $G_{R}$ is of class $C^{1}$ and, since $\mathrm{D}G_{R}
(u)=\mathrm{D}J_{\Omega_{R}}(v_{u})\circ\mathrm{D}\psi_{R}(u),$ we have that
$\mathrm{D}J_{\Omega_{R}}(v_{u})w=0$ for every $w\in T_{v_{u}}\tilde{\Sigma
}_{R}$ if $u$ is a critical point of $G_{R}.$ But $v_{u}$ was chosen so that
$\mathrm{D}J_{\Omega_{R}}(v_{u})z=0$ for every $z\in T_{u}^{\perp}\Sigma_{R}.$
Hence, $v_{u}$ is a critical point of $J_{\Omega_{R}}$ if $u$ is a critical
point of $G_{R}.$
\end{proof}
\subsection{The proof of Theorems \ref{thm:no-boundary} and
\ref{thm:with-boundary}}
\label{sec:find-locat-crit}
From Propositions \ref{prop:lyapunov-gradient-estimate},
\ref{prop:jr-high-on-boundary} and \ref{prop:jr-low-in-interior} if
$\gamma(0)=\gamma(1)$, or from Propositions
\ref{prop:lyapunov-gradient-estimate-open},
\ref{prop:jr-high-on-boundary-open} and \ref{prop:jr-low-in-interior-open} if
$\gamma(0)\neq\gamma(1)$, and estimate \eqref{eq:31}, we obtain
\begin{align*}
\min_{X\in\partial\mathcal{U}_{1,R}}G_{R}(\varphi_{R}(X))\geq &  E_{n}
+\beta\mathrm{e}^{-\mu g_{1}(R)}+o(\mathrm{e}^{-\mu g_{1}(R)})\\
&  +O(R^{-2/3})+O(\mathrm{e}^{-2\alpha^{\prime}\mu g_{2}(R)})\\
\min_{X\in\overline{\mathcal{U}_{1,R}}}G_{R}(\varphi_{R}(X))\leq &
E_{n}+o(\mathrm{e}^{-\mu g_{1}(R)})+O(R^{-2/3})+O(\mathrm{e}^{-2\alpha
^{\prime}\mu g_{2}(R)}).
\end{align*}
We set
\begin{equation*}
g_{1}(R):=\frac{1}{2\mu}\log R\qquad\text{and}\qquad g_{2}(R):=\left(
\frac{1}{2}+\frac{1}{4\alpha^{\prime}}\right)  g_{1}(R).
\end{equation*}
Since $\alpha^{\prime}>1/2,$ these functions satisfy \eqref{eq:24},
\eqref{eq:10}, and \eqref{eq:25}. Note that
\begin{equation*}
R^{-2/3}=o(\mathrm{e}^{-\mu g_{1}(R)})\qquad\text{and}\qquad\mathrm{e}
^{-2\alpha^{\prime}\mu g_{2}(R)}=o(\mathrm{e}^{-\mu g_{1}(R)}).
\end{equation*}
Therefore,
\begin{equation*}
\min G_{R}(\varphi_{R}(\overline{\mathcal{U}_{1,R}}))<\min G_{R}(\varphi
_{R}(\partial\mathcal{U}_{1,R}))
\end{equation*}
if $R$ is large enough. It follows that $G_{R}$ has a local minimum
$w_{R}:=\varphi_{R}(X_{R})$ in $\varphi_{R}(\mathcal{U}_{1,R})\subset
\Sigma_{R}$. Hence, by Lemma~\ref{lem:reduced-problem}, $u_{R}:=v_{w_{R}}$ is
a critical point of $J_{\Omega_{R}}$.
Moreover, by \eqref{eq:30}, we have that $u_{R}=\varphi_{R}(X_{R})+o(1)$ in
$H_{0}^{1}(\Omega_{R})$ as $R\rightarrow\infty$. This, together with estimates
\eqref{eq:16}, \eqref{eq:28} and \eqref{eq:44}, yields \eqref{eq:4} and \eqref{eq:5}.
Finally, \eqref{eq:10} implies that $\lvert x_{R,i}-x_{R,j}\rvert
\rightarrow\infty$ if $i\neq j$ and that dist$(x_{R,i},\partial\Gamma
_{R})\rightarrow\infty$ for all $i$, as $R\rightarrow\infty$. The proofs of
Theorems \ref{thm:no-boundary} and \ref{thm:with-boundary} are complete.
\ \hfill$\square$
\bibliographystyle{amsplain-abbrv}
\bibliography{aclapa-1}
\end{document}